\documentclass[11pt]{article}

\usepackage{geometry}
\usepackage{listings}
\usepackage{cite} 
\usepackage{amsthm}
\usepackage{amsfonts}
\usepackage{graphicx}
\usepackage{epstopdf}
\usepackage{dsfont}
\usepackage{bm}
\usepackage{bbm}
\usepackage{tikz}
\usepackage{pgfplots}
\usetikzlibrary{patterns}
\usepackage{ulem}
\usetikzlibrary{shapes}
\usetikzlibrary{plotmarks}

\usepackage[numbers]{natbib}  
\usepackage{amsmath}
\usepackage{amssymb}
\usepackage{graphicx}
\usepackage{hyperref}

\usepackage{tikz}
\usepackage{pgfplots}
\usetikzlibrary{shapes}
\usepackage{tikz-3dplot}
\usetikzlibrary{decorations.pathreplacing,angles,quotes,patterns}
\usetikzlibrary{decorations.pathmorphing}
\tikzset{snake it/.style={decorate, decoration=snake}}
\usetikzlibrary{positioning}
\usetikzlibrary{patterns}
\usepackage{subfigure}
\usetikzlibrary{patterns,positioning}
\usetikzlibrary{intersections}
\usetikzlibrary{backgrounds,patterns}
\usetikzlibrary{tikzmark,calc,fadings,arrows,shapes,decorations.pathreplacing}
\usetikzlibrary{shapes,positioning,shadings}
\usepgfplotslibrary{fillbetween}

\usepackage{caption}
\usepackage{chngcntr}
\counterwithout{figure}{section}

\usepackage{color}
\usepackage{todonotes}

\usepackage{amsopn}
\RequirePackage{amsmath}
\RequirePackage{fix-cm}

\usepackage{graphicx}
\usepackage{subfigure}
\usepackage{epstopdf}
\usepackage{indentfirst}
\usepackage{color}
\usepackage{bm}
\usepackage{tikz}
\usetikzlibrary{calc}
\usepackage{ifthen}

\usepackage{ulem}

\numberwithin{equation}{section} 

\newcommand\ds{\displaystyle}

\def\on{{\bf 1}}
\def\O{\Omega}
\def\o{\omega}
\def\G1{{\bf 1}}

\def\S{{\GS}}
\def\d{\delta}

\def\e{\varepsilon}

\def\fu{\frak{u}}
\def\wo{\overline}
\def\C{\mathcal{C}}
\def\BE{{\bf E}}

\def\Te{{\cal T}_\e}

\definecolor{navyblue}{rgb}{0.0, 0.0, 0.5}

\def\Ga{{\bf a}}
\def\Gb{{\bf b}}
\def\Ge{{\bf e}}

\def\ds{\displaystyle}

\def\N{{\mathbb{N}}}

\def\R{{\mathbb{R}}}
\def\Z{{\mathbb{Z}}}

\def\D{{\mathbb{D}}}
\def\R{{\mathbb{R}}}

\def\N{{\mathbb{N}}}
\def\Z{{\mathbb{Z}}}

\def\Ec{{\mathcal{E}}}

\def\Pc{{\mathcal{P}}}
\def\Kc{{\mathcal{K}}}
\def\Uc{{\mathcal{U}}}
\def\Rc{{\mathcal{R}}}

\def\Vc{{\mathcal{V}}}

\def\Yc{{\mathcal{Y}}}
\def\Pc{{\mathcal{P}}}
\def\Wc{{\mathcal{W}}}

\def\Mc{{\mathcal{M}}}

\def\wh{\widehat }

\def\wt{\widetilde }
\def\X{\times }

\def\Ga{{\bf a}}
\def\Gb{{\bf b}}

\def\Ge{{\bf e}}

\def\Gu{{\bf u}}
\def\Gv{{\bf v}}
\def\Gw{{\bf w}}

\def\GA{{\bf A}}
\def\GB{{\bf B}}
\def\GC{{\bf C}}
\def\GD{{\bf D}}

\def\GF{{\bf F}}
\def\GG{{\bf G}}
\def\GH{{\bf H}}
\def\GI{{\bf I}}

\def\GM{{\bf M}}

\def\GS{{\bf S}}

\def\GU{{\bf U}}
\def\GF{{\bf F}}
\def\GV{{\bf V}}

\definecolor{skyblue}{RGB}{135,206,235}
\definecolor{deepskyblue}{RGB}{0, 191, 255}

\usepackage{amsfonts}
\usepackage{graphicx}
\usepackage{epstopdf}
\usepackage{dsfont}
\usepackage{bm}
\usepackage{bbm}
\usepackage{algorithmic}
\usepackage{tikz}
\usepackage{pgfplots}
\usetikzlibrary{patterns}

\usepackage{nicefrac}
\usepackage{mathtools}

\usepackage{ulem}

\usepackage{url, breakurl}
\usepackage{charter} 
\usepackage{longtable}
\usepackage{float}
\usepackage{supertabular}
\usepackage{graphicx}
\usepackage{color}
\usepackage{subfigure}
\usepackage{overpic}
\usepackage{textcomp}
\usepackage[section]{placeins}

\usepackage{enumitem}

\usepackage{siunitx}
\newtheorem{theorem}{Theorem}
\newtheorem{definition}{Definition}

\newtheorem{remark}{Remark}

\newtheorem{lemma}{Lemma}

\begin{document}


\title{Mathematical modelling and homogenization of thin fiber-reinforced hydrogels}
	
	
	%
	\author{Amartya Chakrabortty, Haradhan Dutta, Hari Shankar Mahato}
	
	\maketitle
	
	{\bf { Mathematics Subject Classification (2010):} 35B27, 35B40, 35C20, 74Q15, 74S20, 74A15.}
	{\let\thefootnote\relax\footnotetext{
			\noindent Amartya Chakrabortty: SMS, Fraunhofer ITWM, Kaiserslautern 67663, Germany
            
            Email: amartya.chakrabortty@itwm.fraunhofer.de\\
            
		Haradhan Dutta, Hari Shankar Mahato: Department of Mathematics, IIT Kharagpur, India
        
		Emails: hdutta1412@kgpian.iitkgp.ac.in, hsmahato@maths.iitkgp.ac.in\\
            
            Corresponding author email: amartya.chakrabortty@itwm.fraunhofer.de
            }}







\begin{abstract}
This work considers simultaneous homogenization and dimension reduction of a poroelastic model for thin fiber-reinforced hydrogels. The analysed medium is a two-component system consisting of a continuous fiber framework with hydrogel inclusions arranged periodically throughout the medium. The fibers are assumed to operate under quasi-stationary linear elasticity, whereas the hydrogel's hydromechanical behavior is represented using Biot's linear poroelasticity model. The asymptotic limit of the coupled system is established when the periodicity and thickness parameters are of the same order and tend to zero simultaneously, utilizing the re-scaling unfolding operator. It is demonstrated that the limit displacement exhibits Kirchhoff-Love-type behavior using the decomposition of plate displacements. Towards the end, a unique solution for the macroscopic problem has been demonstrated.
\end{abstract}


	
\section{Introduction}
 Fiber-reinforced hydrogels (FIH) are gaining prominence due to their biocompatibility and ability to mimic the mechanical properties of biological tissues. They are being explored for applications in tissue engineering \cite{FIH1}, drug delivery systems, and regenerative medicine \cite{FIH2}. The fibrous reinforcement provides the necessary mechanical support for cell growth and differentiation, while the hydrogel matrix facilitates nutrient transport and waste removal. This combination is particularly advantageous in developing scaffolds for soft tissue repair, where the mechanical properties must closely match those of the natural tissue to ensure successful integration and functionality.

Hydrogels are made up of polymer chains saturated with water, and they have complex mechanical properties. A substantial amount of previous research exists concerning the modeling of hydrogels, with a few notable studies highlighted here, cf. \cite{EMHydrogel, Chen1, Castilho1, SM02, SMPB23} and references therein.  In \cite{Castilho1}, hydrogels are modeled as hyper-elastic materials (Neo-Hookean solids) and in \cite{Mix01}, hydrogels are modeled as a mixture governed by Mixture theories. In the present work, it is modeled as in \cite{Chen1}, \cite{EMHydrogel} and \cite{Same1} using Biot-linear poroelasticity, which accounts for the influence of pore fluid diffusion on the quasi-static displacement of the porous media. Biot’s theory treats the fluid-solid mixture as a single homogenized continuum body, and it can be deduced as a special case of the theory of linear chemo-elasticity discussed in Chapter 15 of \cite{Anand01}. It is also experimentally seen in \cite{TvEBiot} that Biot's law,  which allows for a mass flux of the fluid, and not as a multi-component mixture as in the “theory of mixtures” (see \cite{Mix01}), seems to give more accurate properties of hydrogel. To the best of our knowledge, a mathematical analysis of the effective material properties of hydrogels does not exist.

In general, a FIH consists of a periodic fiber scaffold with the inclusion of hydrogels. Usually, both the fiber and hydrogel parts are connected (see Figure 1 in \cite{Castilho1} for an electron microscopy image of a printed fiber-scaffold and hydrogel). In this paper, as in \cite{Chen1}, \cite{SMPB23} and \cite{EMHydrogel}, we assume that the fiber-scaffold is closed, in particular, the fiber-scaffold is connected, but the hydrogel part is disconnected, and they interact with each other via the fiber-scaffold\footnote{In the Conclusion and outlook section we comment about the case when both the hydrogel and fiber part is connected.}.

In practice, the filament spacing of the scaffold is usually in the range of $\mu$m while the overall size of an FIH is in the range of mm to cm. Due to this scale heterogeneity, the effective hydro-mechanical properties of a FIH is not yet fully understood, and as a consequence, there is an interest in finding the effective global properties from its microstructure using homogenization. Furthermore, in reality, the fiber scaffold is very thin. So, it is interesting to consider an additional parameter for thickness along with the individual cell size.\\[1mm]
In this study, we present a framework for simultaneous homogenization and dimension reduction of a thin FIH in the context of quasi-stationary linear elasticity coupled with Biot's linear poroelasticity. The domain under consideration is a periodic connected plate $\O^f_\e$ made of linear elastic fibers with holes filled by hydrogels $\O^g_\e$, modelled via Biot's linear poroleasticity (See Fig \ref{Fig01}), with each periodicity cell isometric to $(0,\e)^2\X (-\e,\e)$. Using the unfolding operators, the homogenized (limit) system is obtained and is shown to be of Biot-Kirchhoff-Love poroelastic plate type.\\[1mm]
 An effective model addressing the two-phase elastic poroelastic micro-scale problem is derived in \cite{Chen1} through formal asymptotic expansion. At the macroscale, i.e., the domain without any heterogeneities, recently several works have been done for poro-elastic and poro-visco-elastic, e.g., \cite{BS22, BGSP22, BW21, BGSW16, BMW23}.\\[1mm]
For simultaneous homogenization dimension reduction, re-scaling of the unfolding operator is used, which is a variation of the periodic unfolding operator, introduced in \cite{CDG2} and further developed in \cite{CDG1}. For a more comprehensive understanding of homogenization techniques, please refer to \cite{tartar1}, \cite{asymptotic1}, \cite{Homogenization1},  and \cite{twoscale1}. Moreover, for further literature on homogenization and dimension reduction in the linear regime, consult \cite{MGahn01}, \cite{Larysa01}, \cite{GJMContact}, \cite{hauck2}  and \cite{GOW}. The approach presented in this paper, which involves an extension operator and simultaneous homogenization dimension reduction using the re-scaling unfolding operator, has similarities to the approach presented in \cite{Amartya}. The key tool used to derive the exact limit homogenized problem and to show that the limit displacement is of Kirchhoff-Love type is Griso's decomposition of plate displacement. This kind of decomposition for displacements was first introduced in \cite{BGRod1} and further developed in \cite{GKL} and \cite{GrisoRodB}. For general references on elasticity, see \cite{elasticityO}--\cite{elasticity} and mathematical modelling of Biot linear poroelasticity, see \cite{Anand01}.\\[1mm]
The Barenblatt-Biot model's mathematical formulation and analysis are explored in \cite{SM02}. The model in this context is similar to the micro-scale thin model. The reduced model for Biot poroelasticity is provided in \cite{Mikelic01}, utilizing dimension reduction, and it is demonstrated that the Biot-Kirchhoff-Love poroelastic plate model is obtained as the limit model. The present work generalizes this result through simultaneous homogenization dimension reduction of a highly heterogeneous coupled media. Additionally, in \cite{MGahnBiot} and \cite{Igor1}, a simultaneous homogenization dimension reduction of the quasi-static Stokes equation coupled with a linear elastic solid is presented, specifically addressing cases where the thickness and period are of the same order, as well as scenarios where the period is smaller than the thickness, respectively.\\[1mm]
In \cite{EMHydrogel}, the authors have considered a microscopic model analogous to that of the current work in $\R^2$ or $\R^3$, excluding the thickness parameter. They derived an effective macroscopic model utilizing the two-scale convergence method for at least a subsequence. The existence of a unique solution with $L^2$ regularity in time for the microscopic problems was demonstrated using the fixed-point theorem for monotone operators. In the present work, the Galerkin method is employed to establish the existence of a unique solution with $H^1$ regularity in time for the microscopic problem. The inclusion of the thickness parameter necessitates a comprehensive reassessment of the analysis, beginning with the derivation of estimates. Ultimately, the limit homogenized problem is derived, and the existence of a unique solution is provided.\\[1mm]
The paper is organized as follows: General notations are introduced in Section \ref{N-1}. Subsequently, Section \ref{Sec03} presents a description of the domain and the coupled system in both strong and weak forms. The main results are summarized in Section \ref{Sec04}. Preliminary results regarding the extension and decomposition of plate displacement are provided in Section \ref{Sec05}. In Section \ref{Sec06}, the existence of the microscopic problem and a priori estimates are derived. An asymptotic analysis of the microscopic fields is conducted in Section \ref{Sec07}, and finally, the macroscopic problem and the existence of its solution are demonstrated in Section \ref{Sec08}.
\section{General notations} \label{N-1}
Throughout the paper, the following notation will be used:
\begin{itemize}
	\item  $\e\in \R^+$ are small parameters;
	\item $\Yc:=Y\times (-1,1 )$ be the reference cell of the structure, where $Y=(0,1)^2\subset \R^2$ bounded domain;
	\item $x=(x',x_3)=(x_1,x_2,x_3)\in \R^3$, $\ds \partial_i\doteq\frac{\partial}{\partial x_i}$ denote the partial derivative w.r. to $x_i$, $i\in \{1,2,3\}$;
	\item the mapping $u:\O_\e\to\R^3$ denotes the displacement arising in response to forces. The gradient of displacement is denoted by $\nabla u$, with $ \nabla u\in \GM_3$ for every $x\in \O_\e$ where $\GM_3$ is the space of $3\times 3$  real matrices;
	\item  for a.e. $x'\in\R^2$, 
	$$ x'=\left(\e\left[{x'\over \e}\right]+\e\left\{{x'\over\e}\right\}\right),\quad \text{where}\;\; \left[{x'\over \e}\right]\in\Z^2,\;\;\left\{{x'\over\e}\right\}\in Y;$$
	\item  $|\cdot|_F$ denote the Frobenius norm on $\R^{3\X3}$;
	\item  for every displacement $u\in H^1(\O_\e)^3$, the linearized strain tensor is given by
	$$e(u)={1\over 2}\big(\nabla u+(\nabla u)^T\big)\quad \text{and}\quad e_{ij}(u)={1\over 2}\left({\partial u_i\over \partial x_j}+{\partial u_j\over \partial x_i}\right)$$
	where $i,j\in \{1,2,3\}$;
	\item 	for any $\psi\in H^1(\Yc)^3$, (resp. $L^2\big(\o; H^1(\Yc)\big)^3$) we denote $e_y(\psi)$ the $3\X 3$ symmetric matrix whose entries are
	$$e_{y,ij}(\psi)={1\over 2}\left({\partial \psi_i\over \partial y_j}+{\partial \psi_j\over \partial y_i}\right);$$
	\item $(\alpha,\beta)\in \{1,2\}^2$, $(i,j,k,l)\in \{1,2,3\}^4$ and $a\in\{f,g\}$ (if not specified).
    \item  By convention in estimates we simply write $L^2(\O_{\e})$ instead of $L^2(\O_{\e})^3$ or $L^2(\O_{\e})^{3\times 3}$, we write the complete space when we give weak or strong convergence.
\end{itemize}
{ 	In this paper we use the Einstein convention of summation over repeated indices and denote by $C$ a generic constant independent of $\e$. These are some of the general notations used in this paper, notations which are not defined here are defined in the main content of the paper.}

\section{Problem setup and domain description}\label{Sec03}
\subsection{Domain description}
The whole system of the fiber-reinforced hydrogels is a periodic perforated plate with holes filled with hydrogels. It is described as follows:
 \begin{itemize}
 	\item Let $\o\subset \R^2$ be a bounded domain with Lipschitz boundary.  Set 
 	$$\Kc_\e=\{k\in\Z^2\,|\,\e k+\e Y\subset \o\},\quad \o_\e=\text{interior}\left\{\bigcup_{k\in\Kc_\e}\e(k+\wo{Y})\right\},$$
 	where the set $(\o\setminus\overline{\o_\e})$ contains the part of the cells intersecting the boundary $\partial\o$. Observe that since $\o$ is bounded with Lipschitz boundary $|(\o\setminus\overline{\o_\e})|\to0$ as $\e\to0$. So, the behavior of limit fields comes from $\o_\e$.
 	\item Let $Y^f\subset Y$ is an bounded domain with Lipschitz boundary such that interior$\big({\wo Y^f}\cup({\wo Y^f}+\Ge_\alpha)\big)$ is connected for $\alpha=1,2$ and $Y^g=Y\setminus {\wo Y^f}$.
 	\item $\Yc^f$ and $\Yc^g$ are the reference cell of the fiber part and the gel part respectively, where $\Yc^a=Y^a\X(-1,1)$ for $a\in\{f,g\}$.
 	\item $\Yc=Y\X(-1,1)$ is the reference cell of the whole system. Observe that $\Yc=\Yc^f\cup \Yc^g\cup \Gamma$, where $\Gamma=\gamma\X(-1,1)=({\wo Y^f}\cap{\wo Y^g})\X(-1,1)$.
 	\item We define 
 	$$\o_\e^g=\bigcup_{k\in\Kc_\e}\e(Y^g+k),\quad \o_\e^f=\o\setminus \overline{\o^g_\e}\quad \text{and}\quad \gamma_\e=\bigcup_{k\in\Kc_\e}\e(\gamma+k).$$
 	\item So, the whole system is denoted by $\O_\e=\o\X(-\e,\e)$ with the fiber and gel part is denoted by
 	$$\O^a_\e=\o_\e^a\X(-\e,\e)\quad \text{and}\quad \Gamma_\e=\gamma_\e\X(-\e,\e).$$ 
 	\item Denote $\Lambda_{0,\e}=\partial\o\X(-\e,\e)$. We can replace $\partial\o$ by $\Lambda\subset \partial\o$ with non-zero measure.
 	\item   We set $n_\e=n_\e(x)$, $x\in\Gamma_\e$ the normal vector of $\Gamma_\e$ pointing outwards of $\O^g_\e$.
 	\item $\on^a_\e:\O_\e\to\{0,1\}$ for $a\in\{f,g\}$ denote the characteristics functions corresponding to $\O_\e^a$.
 	\item $(0,T)$ where $T>0$ represents the time interval.
 \end{itemize}
Observe that $\O^f_\e$ is connected and $\O^g_\e$ is disconnected.\\
\begin{figure}[ht]
    \centering
    \subfigure[$2$D-view of a cell and the domain]{
\begin{tikzpicture}[scale=0.8]

\fill[black!30, opacity=0.5] (0,0) rectangle (4,4); 
\draw[thick] (0,0) rectangle (4,4); 
\fill[navyblue] (1,1) rectangle (3,3); 

\begin{scope}[shift={(5,0)}] 
    \fill[black!30, opacity=0.5] (0,0) rectangle (4,4); 
    \draw[thick] (0,0) rectangle (4,4); 

    \foreach \x in {0,1,2,3} { 
        \foreach \y in {0,1,2,3} { 
            \fill[navyblue] (\x+0.2,\y+0.2) rectangle (\x+0.8,\y+0.8); 
        }
    }
\end{scope}
\end{tikzpicture}}
\hspace{2.5cm}
 \subfigure[$3$D-view of a cell]{
 \begin{tikzpicture}[scale=0.80]
    \fill[black!20] (1.5, 1.5, 1.5) -- (1.5, -1.5, 1.5) -- (-1.5, -1.5, 1.5) -- (-1.5, 1.5, 1.5) -- cycle; 
    \fill[black!20] (1.5, 1.5, -1.5) -- (1.5, -1.5, -1.5) -- (-1.5, -1.5, -1.5) -- (-1.5, 1.5, -1.5) -- cycle; 
    \fill[black!20] (1.5, 1.5, 1.5) -- (1.5, 1.5, -1.5) -- (1.5, -1.5, -1.5) -- (1.5, -1.5, 1.5) -- cycle; 
    \fill[black!20] (-1.5, 1.5, 1.5) -- (-1.5, 1.5, -1.5) -- (-1.5, -1.5, -1.5) -- (-1.5, -1.5, 1.5) -- cycle; 
    \fill[black!20] (1.5, 1.5, 1.5) -- (1.5, -1.5, 1.5) -- (-1.5, -1.5, 1.5) -- (-1.5, 1.5, 1.5) -- cycle; 
    \fill[black!20] (1.5, 1.5, -1.5) -- (1.5, -1.5, -1.5) -- (-1.5, -1.5, -1.5) -- (-1.5, 1.5, -1.5) -- cycle; 

    \fill[navyblue] (0.5, -1.5, 0.5) -- (0.5, -1.5, -0.5) -- (-0.5, -1.5, -0.5) -- (-0.5, -1.5, 0.5) -- cycle; 
    \fill[navyblue] (0.5, 1.5, 0.5) -- (0.5, 1.5, -0.5) -- (-0.5, 1.5, -0.5) -- (-0.5, 1.5, 0.5) -- cycle; 
    \fill[navyblue] (0.5, -1.5, 0.5) -- (0.5, 1.5, 0.5) -- (0.5, 1.5, -0.5) -- (0.5, -1.5, -0.5) -- cycle; 
    \fill[navyblue] (-0.5, -1.5, 0.5) -- (-0.5, 1.5, 0.5) -- (-0.5, 1.5, -0.5) -- (-0.5, -1.5, -0.5) -- cycle; 
    \fill[navyblue] (0.5, 1.5, 0.5) -- (0.5, -1.5, 0.5) -- (-0.5, -1.5, 0.5) -- (-0.5, 1.5, 0.5) -- cycle; 
    \fill[navyblue] (0.5, -1.5, -0.5) -- (0.5, 1.5, -0.5) -- (-0.5, 1.5, -0.5) -- (-0.5, -1.5, -0.5) -- cycle; 
\end{tikzpicture}}

    \caption{$\O_\e$ the domain. $\O^g_\e$ and $\O^f_\e$ are denoted by blue and light black regions, respectively.}
    \label{Fig01}
\end{figure}

\subsection{The poro-elasticity problem}
The coupled system is governed by the quasi-stationary linear elasticity in the fiber part and Biot's linear poroelasticity in the gel part, which is given by
 \begin{align}
  -\nabla\cdot(A^f_\e e(U_\e))&=f_\e\on^f_\e,\quad &&\text{in}\;\;(0,T)\X\O^f_\e,\\
  -\nabla\cdot (A^g_\e e(U_\e)-\alpha p_\e\GI_3)&=f_\e\on^g_\e\quad&&\text{in}\;\; (0,T)\X\O^g_\e,\\
  \label{Biot1}
  \partial_t(cp_\e+\alpha\nabla\cdot U_\e)-\nabla\cdot (\e^2K\nabla p_\e)&=h_\e\quad &&\text{in}\;\;(0,T)\X\O^g_\e,\\
  	\label{25}
  (A^g_\e e(U_\e)-\alpha p_\e\GI_3)n_\e&=(A^f_\e e(U_\e))n_\e\quad &&\text{on}\;\; (0,T)\X \Gamma_\e,\\
  -\e^2K\nabla p_\e\cdot n_\e&=0,\quad &&\text{on}\;\;(0,T)\X\Gamma_\e,\\
  	\label{28}
  U_\e&=0,\quad &&\text{on}\;\;(0,T)\X\Lambda_{0,\e},\\
  p_\e(0)&=0,\quad &&\text{in}\;\;\O^g_\e,\\
  \label{210}
  w_\e(0)&=0,\quad &&\text{in}\;\;\O^f_\e,
 \end{align}
 where $U_\e:\O_\e\to\R^3$ represents the displacement of the fiber-reinforced hydrogel (i.e. for both the fiber and gel part) with $p_\e:\O_\e^g\to\R^3$ denoting the pore-pressure.
$A_\e^a$ are the elasticity tensors for the fiber and gel part, respectively. $c$, $\alpha$, and $\e^2K$ (symmetric and positive-definite with constant $c_K>0$) are the Biot modulus, the Biot-Willis parameter and the permeability matrix. The external forces are given by $f_\e$ and $h_\e$. 		We have the following assumption on the external loading(s):
\begin{itemize}
	\item Let $f\in H^1([0,T];L^2(\o)^3)$ and $h\in L^2(S\X\o^g)$ such that there exist a $K_1>0$ independent of $\e$ satisfying 
	\begin{equation}\label{AF01}
		\|f\|_{H^1((0,T);L^2(\o))}+\|h\|_{L^2((0,T)\X\o^g)}\leq K_1.
	\end{equation}
	\item Let for
	\begin{equation}\label{AF02}
		\begin{aligned}
			f_\e&=\e f_{\alpha}\Ge_\alpha+\e^2 f_{3}\Ge_3,\quad&&\text{in $(0,T)\X\O_\e$},\\
			h_\e&= \e h,\quad&&\text{in $(0,T)\X\O^g_\e$}.		
		\end{aligned}
	\end{equation}
\end{itemize}
We also assume that the $4$th order Hooke's elastic tensors $A^a_{ijkl}\in L^\infty(\Yc)$ and satisfy the following for $a\in\{f,g\}$
\begin{itemize}
	\item Symmetric: $A^a_{ijkl}=A^a_{jikl}=A^a_{klij}$,
	\item Coercivity condition ($c_0>0$) for all $3\X3$ symmetric matrix $S$ : 
	\begin{equation}\label{CoeCon}
		A^a_{ijkl}(y)S_{ij}S_{kl}\geq c_0S_{ij}S_{ij}=c_0|S|_F^2\quad\text{for a.e. $y\in \Yc^a$}.
	\end{equation} 
	\item We set
	$$A^a_{\e,ijkl}=A^a_{ijkl}\left(\left\{{x'\over \e}\right\},{x_3\over \e}\right),\quad\text{for a.e. $x\in\O^a_\e$.}$$
	\item Let
	$$\GA=\on_{\Yc^f}A^f+\on_{\Yc^g}A^g,\quad \text{and}\quad \GA_\e=\on_{\O_\e^f}A_\e^f+\on_{\O_\e^g}A_\e^g.$$
\end{itemize}
\begin{remark}
	Observe that both the displacement and the forces are continuous across the interface $\Gamma_\e$. Moreover, there is no gel flux between the different cells.
\end{remark}
%
%
%
%
%
        The  corresponding variational form of the coupled system is given by: Find $(U_\e,p_\e)\in \GU_\e$ such that
		\begin{equation}\label{MainPro}
			\begin{aligned}
				\int_{\O_\e}\GA_\e e(U_\e):e(v)\,dx-\int_{\O_\e^g}\alpha p_\e\nabla\cdot v\,dx=\int_{\O_\e}f_\e\cdot v\,dx,\\
				\int_{\O_\e^g}\partial_t(cp_\e+\alpha\nabla\cdot U_\e)\phi\,dx+\int_{\O_\e^g}\e^2K\nabla p_\e\cdot \nabla\phi\,dx=\int_{\O_\e^g}h_\e\phi\,dx,
			\end{aligned}
		\end{equation}
		for all $(v,\phi)\in H^1_0(\O_\e)^3\X H^1(\O_\e^g)$ and for a.e. $t\in (0,T)$, where 
        \begin{multline*}
            \GU_\e=\big\{(w,p)\in H^1((0,T);H^1(\O_\e)^3\X H^1(\O_\e^g))\;|\;
            \text{such that $w$ and $p$ satisfies \eqref{28}--\eqref{210}}\big\}.
        \end{multline*}	

        \section{Main results}\label{Sec04}
            In this section, we summarize the main results of the paper. For detailed proofs and explanations, see Sections \ref{Sec06}--\ref{Sec08}. 
            \begin{theorem}\label{Th01}
                The $\e$-dependent microscopic problem \eqref{MainPro} has a unique weak solution $(U_\e,p_\e)\in\GU_\e$ for every $\e>0$.
            \end{theorem}
            Specifically, see Section \ref{Sec06} for the proof of the theorem.\\
            The sequence of weak solutions $(U_\e,p_\e)$ convergences in the sense of unfolding operator (See \eqref{LIC01} for the definition of $\D_0$ and Section \ref{Sec07} for the definition of the unfolding operators) to $(\Wc,p_0)\in \D_0\X H^1((0,T); L^2(\o\; H^1_{per}(\Yc^g))$, which also satisfy the following initial conditions
            $$\Wc(0)=0,\quad\text{in $\o$}\quad \text{and}\quad p_0(0)=0,\quad\text{in $\o\X \Yc^g$}.$$ 
            In the convergences \eqref{IL02}, the {\bf Kirchhoff-Love} type limit displacement is derived.  Finally, we have the following homogenized problem (See \eqref{CFeu1} for the definition of $E(\cdot)$):
            \begin{theorem}
                The limit field $(\Wc,p_0)\in \D_0\X H^1((0,T); L^2(\o\; H^1_{per}(\Yc^g))$ is the unique solution of the problem:
                \begin{equation*}
                \left.\begin{aligned}
                    &\int_{\o} \GA^{hom}E(\Wc):E(\Vc)\,dx'
        -\int_{\o}\alpha p_m\nabla\cdot \Vc_L\,dx'\\
        &\hskip 12mm=\int_\o f\cdot \Vc\,dx'-\int_{\o\X \Yc}\GA e_y(\Gu_p):E(\Vc)\,dx'dy,\\
        &\int_{\o\X\Yc^g}\partial_t\big(p_0+\alpha\nabla\cdot\Wc_L\big)\phi\,dx'dy
       +\int_{\o\X\Yc^g}K\nabla_y p_0\cdot\nabla_y\phi\,dx'dy\\
       &\hskip 12mm=\int_{\o\X \Yc^g}h\phi\,dx'dy-\alpha\int_{\o\X \Yc^g}\big(\nabla_y\cdot\partial_t\Gu_d-\nabla_y\cdot\partial_t\Gu_p)\,dx'dy
                \end{aligned}\right\},
                \end{equation*}
                $\forall\,(\Vc,\phi)\in \D_0\X L^2(\o; H^1_{per}(\Yc^g))$,
                where the homogenized tensor $\GA^{hom}$ is given by \eqref{HomC02}--\eqref{HomC01}, $\Gu_d$ is given by \eqref{MD01} and $\Gu_p$ is the unique solution of the cell problem
                $$ \int_{\Yc}\GA(y) e_y(\Gu_p) : e_y(\wo{\Gv})\,dy={1\over |\Yc|}\int_{\Yc^g}\alpha p_0\nabla_y\cdot \wo{\Gv}\,dy,\quad \text{for all $\wo{\Gv}\in H^1_{per,0}(\Yc)^3$},
                $$
                with
                $$p_m(t,x')={1\over |\Yc|}\int_{\Yc^g}p_0(t,x',y)\,dy,\quad \Vc_L(t,x')=\sum_{\alpha=1}^2(\Vc_\alpha(t,x')-y_3\partial_\alpha\Vc_3(t,x'))\Ge_\alpha+\Vc_3(t,x')\Ge_3,$$
                for $x'\in\o$ and $t\in (0,T)$.
            \end{theorem}
        See Section \ref{Sec08} for the proof of the above theorem.
        \section{Preliminary results: Extension  and decomposition results}\label{Sec05}
		\subsection{Extension results}
		Below, we first recall a well-known extension result
		\begin{lemma}\label{lem102} There exists an extension operator ${{\bm\Pc}}$ from $H^1(\Yc^f)^3$ into $H^1(\Yc)^3$ satisfying
			$$
			\begin{aligned}
				&\forall\,\Phi\in H^1(\Yc^f)^3,\qquad {\bm\Pc}(\Phi)\in H^1(\Yc)^3,\qquad  \Pc(\Phi)_{|\Yc^f}=\Phi,\\
				&\big\|e\big({\bm\Pc}(\Phi)\big)\big\|_{L^2(\Yc)}\leq C\|e(\Phi)\|_{L^2(\Yc^f)}.
			\end{aligned}
			$$
		\end{lemma}
As a consequence of the above, we have
		\begin{lemma}\label{lem5}
        Let $(U_\e,p_\e)\in \GU_\e$, there exists $\Gw_\e$ and $\Gu_\e$ in $H^1((0,T);H^1(\O_\e)^3)$ satisfying the following: 
            \begin{equation}\label{ExDis}
                U_\e=\Gw_\e+\Gu_\e\quad \text{in $(0,T)\X\O_\e$}.
            \end{equation}
            Moreover, we have  for all $t\in (0,T)$
        \begin{align}
                \label{46}
				\Gw_\e&=0,\quad \text{on $\Lambda_{0,\e}$},\\
    \label{47}
				\Gu_\e&=0,\quad \text{on $\Gamma_\e\cup{\O^f_\e}$},\\
                				\label{48}
				\|e(\Gw_\e)\|_{L^2(\O_\e)}+\|e(\Gu_\e)\|_{ L^2(\O_\e)}&\leq C\|e(U_\e)\|_{L^2(\O_\e)},\\
                \label{410}
				\|\Gu_\e\|_{L^2(\O_\e^g)}+\e\|\nabla \Gu_\e\|_{L^2(\O_\e^g)}&\leq C\e\|e( \Gu_\e)\|_{L^2(\O_\e)}.
			\end{align}
			The constant(s) are independent of $\e$.
		\end{lemma}
  \begin{proof}
  Let $(U_\e,p_\e)\in \GU_\e$.
  	We apply Lemma \ref{lem102} to the function  $U_{k,\e}$ for all $k\in\Kc_\e$ defined by
			$$U_{k,\e}(y)=U_\e\big(\e k+\e y',\e y_3)\qquad \hbox{for a.e. }\; y\in \Yc^f.$$
			Then, after a change of variables we obtain the extension ${\bm\Pc}_\e(U_\e)$ of ${U_\e}_{|\O^f_\e}$.  
      We set
            $${\bm\Pc}_\e(U_\e)=\Gw_\e,\quad\text{and}\quad \Gu_\e=U_\e-\Gw_\e,\quad\text{in $(0,T)\X \O_\e$},$$
            then, we have
            $$\Gu_\e=0,\quad\text{a.e. in $(0,T)\X(\Gamma_\e\cup {\O^f_\e})$}.$$       
            By construction and due to the estimates in Lemma  \ref{lem102} and the above setup we obtain \eqref{46}--\eqref{48}.\\
      Recalling that
      $$\O^g_\e=\bigcup_{k\in\Kc_\e}\e(k+\Yc^g).$$
      We set $v^k_\e(y)=\Gu_\e(\e k+\e y)$ for all $y\in \Yc^g$ and for all $k\in \Kc_\e$, then using \eqref{47}, we have
       $$v^k_\e=0,\quad \text{on $\partial \Yc^g$ for all $k\in \Kc_\e$}.$$ 
      So, using second Korn's inequality, we get 
      $$\|v^k_\e\|_{L^2(\Yc^g)}+\|\nabla_y v^k_\e\|_{L^2(\Yc^g)}\leq C\|e_y(v^k_\e)\|_{L^2(\Yc^g)},\quad\forall\, k\in\Kc_\e.$$
      The constant only depends on $\Yc^g$, not on $\e$.\\
      We obtain after the changing the variable
      $$\nabla_y v^k_\e(y)=\e\nabla \Gu_\e(\e k+\e y),\quad \text{for a.e. $y\in\Yc^g$},$$
      which imply
      $$\e^3\|\Gu_\e\|^2_{L^2(\e k+\e \Yc^g)}+ \e^5\|\nabla\Gu_\e\|_{L^2(\e k+ \e \Yc^g)}\leq C\e^5 \| e(\Gu_\e)\|^2_{L^2(\e k+\e \Yc^g)}.$$
      Hence, adding for all $k\in \Kc_\e$ and dividing by $\e^3$ give \eqref{410}.
  \end{proof}
\subsection{Griso's decomposition of displacement result}
Below, we recall the decomposition of plate displacements, it was first introduced in \cite{BGRod1}, and further developed in \cite{CDG} and \cite{GKL}. For more details on the lemma below see Subsection 11.1.2 and 11.1.3 of \cite{CDG}.
\begin{lemma}\label{DeComKL1} Let $\Gw\in H^1(\O_\e)^3$, then we have
			\begin{equation}\label{PD1}
				\Gw= W_{E} + \wo{w},\qquad \text{  a.e. in }\;\;\O_\e,
			\end{equation} 
			where $W_{E}$ is given by 
			$$
			W_{E} (x) = \left(\begin{aligned}
				&\Wc_1 (x') + x_3 \Rc_1(x')\\
				&\Wc_2 (x') + x_3 \Rc_2(x')\\
				&\hspace{0.9cm}\Wc_3 (x')
			\end{aligned}\right)
			\;\; \text{for a.e.}\;\; x=(x', x_3) \in \O_\e.
			$$
and			\begin{align*}
     \Wc\in H^1(\o)^3,\quad \Rc\in H^1(\o)^2,\quad \wo{w}\in H^1(\O_\e)^3\\
     \text{with}\quad\Uc(x')={1\over 2\e}\int_{-\e}^\e \Gw(x',x_3)\,dx_3,\quad\text{and}\quad \int_{-\e}^\e\wo{w}(\cdot,x_3)\,dx_3=0,\quad \text{for all $x'\in\o$}. 
			\end{align*}
			  Moreover, we have the following estimates 
			\begin{equation}\label{PD2}
				\begin{aligned}
					\|\Wc_\alpha\|_{H^1(\o)}+\e(\|\Wc_3\|_{H^1(\o)}+\|\Rc_\alpha\|_{H^1(\o)}) &\leq {C \over \e^{1/2}} \|e(\Gw)\|_{L^2(\O_\e)},\\
					\|\partial_\alpha\Wc_3+\Rc_\alpha\|_{L^2(\o)} &\leq {C \over \e^{1/2} }\|e(\Gw)\|_{L^2(\O_\e)},\\
					\|\wo{w}\|_{L^2(\O_\e)} +\e\|\nabla\wo{w}\|_{L^2(\O_\e)} &\leq C\e \|e(\Gw)\|_{L^2(\O_\e)}.
				\end{aligned}
			\end{equation}
            If we have $\Gw=0$ a.e. on $\Lambda_{0,\e}$, then we get	
            $$
			\Wc=0,\quad \Rc=0,\quad\text{a.e. on}\quad \partial\o,\quad \text{and}\quad \wo{w}=0\quad\text{a.e. on}\quad \Lambda_{0,\e}.
			$$
        Moreover, we have the following Korn-type inequalities
			\begin{equation}\label{KornPUD1}
				\begin{aligned}
					\|\Gw_1\|_{L^2(\O_\e)}+\|\Gw_2\|_{L^2(\O_\e)}+\e\|\Gw_3\|_{L^2(\O_\e)}&\leq C\|e(\Gw)\|_{L^2(\O_\e)},\\
					\|\nabla \Gw\|_{L^2(\O_\e)}&\leq {C\over\e}\|e(\Gw)\|_{L^2(\O_\e)}.
				\end{aligned}
			\end{equation}
			The constants do not depend on $\e$.
		\end{lemma}
        \begin{remark}
            The displacement $W_E$ is called an elementary displacement, which is a generalization of rigid displacements. The field $\Wc(x')$ is the mid-surface displacement at the point $x'\in\o$ while $x_3\Rc(x')$ represents the small linearized rotation of the fiber $\{x'\}\X(-\e,\e)$. The residual displacement $\wo{w}$ is called the warping, which captures the microscopic behavior of the displacement $\Gw$.
        \end{remark}
		We also recall another decomposition, see Lemma 3.5 in \cite{griso2025} for more details. 
		\begin{lemma}\label{DeCom2}
			Let $\Gu\in H^1(\O_\e)^3$, then we define
			$$\Uc={1\over 2\e}\int_{-\e}^\e \Gu\,dx_3,\quad \text{in $\o$},\quad\text{and}\quad \wo{u}=\Gu-\Uc,\quad\text{in $\O_\e$}.$$
			Then, we have
			$$\Uc\in H^1(\o)^3,\quad \wo{\Gu}\in H^1(\O_\e)^3,\quad\text{with}\quad \int_{-\e}^\e\wo{u}\,dx_3=0,\quad\text{a.e. in $\o$}.$$
			Moreover, it satisfies
			\begin{equation}
				\begin{aligned}
					\|\Uc\|_{L^2(\o)}&\leq {C\over \e^{1/2}}\|\Gu\|_{L^2(\O_\e)},\quad &&\|\partial_\alpha \Uc\|_{L^2(\o)}\leq {C\over \e^{1/2}}\|\nabla\Gu\|_{L^2(\O_\e)},\\
					\|\wo{u}\|_{L^2(\O_\e)}&\leq C\|{\Gu}\|_{L^2(\O_\e)},\quad&&\|\nabla \wo{u}\|_{L^2(\O_\e)}\leq C\|\nabla\Gu\|_{L^2(\O_\e)}.
				\end{aligned}
			\end{equation}
			The constant(s) are independent of $\e$.\\
			Furthermore, if $\Gu=0$ in $\Gamma_\e\cup {\O^f_\e}$, then we have
			$$\wo{u}=0,\quad\text{in $\Gamma_\e\cup {\O^f_\e}$}.$$
		\end{lemma}
        This kind of decomposition is more suited when we already have a Korn-type inequality and we have sharper $L^2$-estimate of $\Gu$ then the $L^2$-estimate of its gradient $\nabla \Gu$.
		\section{Existence result and a priori estimates}\label{Sec06}
		In this section, firstly we give a priori estimates assuming the existence of a unique solution of the $\e$-problem \eqref{MainPro} in $\GU_\e$ then we provide the existence proof using the Galerkin method.
		\begin{lemma}\label{Lem06}
			Let $(U_\e,p_\e)\in\GU_\e$ be the unique weak solution of \eqref{MainPro}, then it satisfy
			\begin{equation}\label{Est01}
				\begin{aligned}
					\|p_\e\|_{L^\infty((0,T);L^2(\O_\e^g))}+\e\|\nabla p_\e\|_{L^2((0,T)\X\O_\e^g)}&\leq C\e^{3/2},\\
					\|e(U_\e)\|_{L^\infty((0,T); L^2(\O_\e))}& \leq C\e^{{3\over 2}}.
				\end{aligned}				
			\end{equation}
			The constant(s) are independent of $\e$.
		\end{lemma}
  \begin{proof}
  We derive the estimates in the subsequent steps:\\
  {\bf Step 1:} Testing the weak form \eqref{MainPro} with $\left(\partial_{t} U_{\e}, p_{\e}\right)$, we get

$$
\begin{aligned}
&\int_{\O_\e}\GA_\e e\left(U_{\e}\right): e\left(\partial_{t} U_{\e}\right) \mathrm{d} x-\int_{\O_{\e}^{g}} \alpha p_{\e} \nabla \cdot \partial_{t} U_{\e} \mathrm{d} x 
= \int_{\O_{\e}} f_\e \cdot \partial_{t} U_{\e} \mathrm{d} x,\\
&\int_{\O_{\e}^{g}} \partial_{t}\left(c p_{\e}+\alpha \nabla \cdot U_{\e}\right) p_{\e} \mathrm{d} x+\int_{\O_{\e}^{g}} \e^{2} K \nabla p_{\e} \cdot \nabla p_{\e} \mathrm{d} x 
=  \int_{\O_{\e}^{g}} h_\e p_{\e} \mathrm{d} x .
\end{aligned}
$$
We integrate the above equations over time to get 
$$
\begin{gathered}
\int_{0}^{t} \int_{\O_\e}\GA_\e e\left(U_{\e}\right): e\left(\partial_{t} U_{\e}\right) \mathrm{d} x \mathrm{~d} s-\int_{0}^{t} \int_{\O_{\e}^{g}} \alpha p_{\e} \nabla \cdot \partial_{t} U_{\e} \mathrm{d} x \mathrm{~d} s=\int_{0}^{t}\int_{\O_{\e}} f_\e \cdot \partial_{t} U_{\e} \mathrm{d} x\mathrm{d} s, \\
\int_{0}^{t} \int_{\O^g_{\e}} \partial_{t}\left(c p_{\e}+\alpha \nabla \cdot U_{\e}\right) p_{\e} \mathrm{d} x \mathrm{d}s+\int_{0}^{t} \e^{2} K\left\|\nabla p_{\e}(s)\right\|_{L^{2}\left(\O_{\e}^{g}\right)}^{2} \mathrm{~d} s=\int_{0}^{t} \int_{\O_{\e}^{g}} h_\e p_{\e} \mathrm{d} x \mathrm{~d} s,
\end{gathered}
$$
which give
$$
\begin{aligned}
& c \int_{0}^{t} \int_{\O_{\e}^{g}} \partial_{t} p_{\e} p_{\e} \mathrm{d} x \mathrm{~d} s+\int_{0}^{t} \int_{\O_\e}\GA_\e e\left(U_{\e}\right): e\left(\partial_{t} U_{\e}\right) \mathrm{d} x \mathrm{~d} s+\e^{2} c_K \int_{0}^{t}\left\|\nabla p_{\e}(s)\right\|_{L^{2}\left(\O_{\e}^{g}\right)}^{2} \mathrm{~d} s \\
&\hskip 15mm \leq \int_{0}^{t} \int_{\O_{\e}^{g}} h_\e p_{\e} \mathrm{d} x \mathrm{~d} s+\int_{0}^{t}\int_{\O_{\e}} f_\e \cdot \partial_{t} U_{\e} \mathrm{d} x\mathrm{d} s,
\end{aligned}
$$
since the matrix $K$ is positive definite with constant $c_K>0$.\\
Using the coercivity \eqref {CoeCon} of $\GA$, we get
\begin{multline*}
c\left\|p_{\e}\right\|_{L^{2}\left(\O_{\e}^{g}\right)}^{2}+c_0\left\|e\left(U_{\e}\right)\right\|_{L^{2}(\O_\e)}^{2}+2 \e^{2} c_{K} \int_{0}^{t}\left\|\nabla p_{\e}(s)\right\|_{L^{2}\left(\O_{\e}^{g}\right)}^{2} \mathrm{~d} s\\
 \leq c\left\|p_{\e}(0)\right\|_{L^{2}\left(\O_{\e}^{g}\right)}^{2}+c_0\left\|e\left(U_{\e}\right)(0)\right\|_{L^{2}(\O_\e)}^{2}+2 \int_{0}^{t} \int_{\O_{\e}^{g}} h_\e p_{\e} \mathrm{d} x \mathrm{~d} s +
 \int_{0}^{t} \int_{\O_{\e}} f_\e \cdot \partial_{t} U_{\e} \mathrm{d} x \mathrm{~d} s
\end{multline*}
Using the fact that $p_{\e}(0)=0$ and $U_{\e}(0)=0$, we obtain
$$
\begin{aligned}
& c\left\|p_{\e}\right\|_{L^{2}\left(\O_{\e}^{g}\right)}^{2}+c_0\left\|e\left(U_{\e}\right)\right\|_{L^{2}(\O_\e)}^{2}+2 \e^{2} c_{K} \int_{0}^{t}\left\|\nabla p_{\e}\right\|_{L^{2}\left(\O_{\e}^{g}\right)}^{2} \mathrm{~d} s \\
 &\hskip 10mm \leq 2 \int_{0}^{t} \int_{\O_{\e}^{g}} h_\e p_{\e} \mathrm{d} x \mathrm{~d} s
+ \int_{0}^{t} \int_{\O_{\e}} f_\e \cdot \partial_{t} U_{\e} \mathrm{d} x \mathrm{~d} s
\end{aligned}$$
{\bf Step 2:} 
		We estimate the right-hand side. First observe that by integration by parts, we have
		\begin{equation}\label{oo01}
			\begin{aligned}
				\int_0^t\int_{\O_\e}f_\e\cdot\partial_tU_\e\,dxds=-\int_0^t\int_{\O_\e}\partial_t f_\e\cdot U_\e\,dxds+\int_{\O_\e}f_\e\cdot U_\e\,dx.
			\end{aligned}
		\end{equation}
			Using the Lemma \ref{lem5}, we have an expression like \eqref{ExDis} for $U_\e$, i.e. $U_\e=\Gw_\e+\Gu_\e$ in $(0,T)\X \O_\e$ with $\Gu_\e=0$ in $(0,T)\X (\Gamma_\e\cup{\O^f_\e})$ and using Lemma \ref{DeComKL1}, we have decomposition for $\Gw_\e$ with the Korn-type estimates \eqref{KornPUD1}. Then, these estimates along with the above equality \eqref{oo01} and the expression \eqref{AF02} for $f_\e$, give 
			$$\begin{aligned}
				-\int_0^t\int_{\O_\e}\partial_t f_\e\cdot U_\e\,&dxds+\int_{\O_\e}f_\e\cdot U_\e\,dx\\
				&=-\int_0^t\int_{\O_\e}\partial_t f_\e\cdot (\Gw_\e+\Gu_\e)\,dxds+\int_{\O_\e}f_\e\cdot (\Gw_\e+\Gu_\e)\,dx\\
				&\leq C\int_0^t\|\e\partial_t f\|_{L^2(\O_\e)}\left(\|\Gw_{\alpha, \e}\|_{L^2(\O_\e)}+\e\|\Gw_{\e,3}\|_{L^2(\O_\e)}+\|\Gu_\e\|_{L^2(\O_\e)}\right)\,ds\\
				&\hskip 15mm+C\|\e f\|_{L^2(\O_\e)}\left(\|\Gw_{\alpha,\e}\|_{L^2(\O_\e)}+\e\|\Gw_{\e,3}\|_{L^2(\O_\e)}+\|\Gu_\e\|_{L^2(\O_\e)}\right)\\
				&\leq C\int_0^t\|\e\partial_t f\|_{L^2(\O_\e)} \big(\|e(\Gw_{\e})\|_{L^2(\O_\e)}+\e\|e(\Gu_\e)\|_{L^2(\O_\e)}\big)\,ds\\
                &\hskip 15mm+C\|\e f\|_{L^2(\O_\e)}\big(\|e(\Gw_\e)\|_{L^2(\O_\e)}+\e\|e(\Gu_\e)\|_{L^2(\O_\e)}\big),
			\end{aligned}$$
			which along with the estimate \eqref{48} give
            \begin{multline*}
                -\int_0^t\int_{\O_\e}\partial_t f_\e\cdot U_\e\,dxds+\int_{\O_\e}f_\e\cdot U_\e\,dx\\
                \leq C\e^{3/2}\int_0^t\|\partial_t f\|_{L^2(\o)} \|e(U_{\e})\|_{L^2(\O_\e)}\,ds+C\e^{3/2}\| f\|_{L^2(\o)}\|e(U_\e)\|_{L^2(\O_\e)}.
            \end{multline*}
            Then, applying Young's inequality give
			\begin{multline*}
				-\int_0^t\int_{\O_\e}\partial_t f_\e\cdot U_\e\,dxds+\int_{\O_\e}f_\e\cdot U_\e\,dx\\
				\leq C\e^{3} \int_0^t\|\partial_tf\|^2_{L^2(\o)}\,ds+C\int_{0}^t\|e(U_\e)\|^2_{L^2(\O_\e)}\,ds
                +C\e^{3}+C\|f\|^2_{L^2(\o)}\|e(U_\e)\|^2_{L^2(\O_\e)}. 
			\end{multline*}
%
   {\bf Step 3:} 	So, the estimates from Steps 1 and 2 along with \eqref{48} and the assumption of forces \eqref{AF01} give
			\begin{multline*}
				c\|p_\e\|^2_{L^2(\O_\e^g)}+c_0\|e(U_\e)\|^2_{L^2(\O_\e)}+2\e^2c_K\int_{0}^t\|\nabla p_\e\|^2_{L^2(\O_\e^g)}ds\\
				\leq C\int_0^t\big(\|e(U_\e)\|^2_{L^2(\O_\e)}+\|p_\e\|^2_{L^2(\O_\e^g)}\big)\,ds+C\|f\|^2_{L^2(\o)}\|e(U_\e)(t)\|^2_{L^2(\O_\e)}\\ +C\e^{3}\int_{0}^t\big(\|h\|^2_{L^2(\o^g)}+\|\partial_t f\|^2_{L^2(\o)}\big)\,ds+C\e^3.
			\end{multline*}
            Using the assumption \eqref{AF01} on forces, we can choose that
            $$C\|f\|^2_{L^2(\o)}\leq CK_1\leq {c_0\over 2},$$
            then, we have used \eqref{AF01} to obtain
            \begin{multline*}
				c\|p_\e\|^2_{L^2(\O_\e^g)}+c_1\|e(U_\e)\|^2_{L^2(\O_\e)}+2\e^2c_K\int_{0}^t\|\nabla p_\e\|^2_{L^2(\O_\e^g)}ds\\
				\leq C\int_0^t\big(\|e(U_\e)\|^2_{L^2(\O_\e)}+\|p_\e\|^2_{L^2(\O_\e^g)}\big)\,ds+C_T\e^3.
			\end{multline*}
            Let us set
$$
\Phi(t) := \|p_\varepsilon(t)\|^2_{L^2(\Omega^g_\varepsilon)} + \|e(U_\varepsilon(t))\|^2_{L^2(\Omega_\varepsilon)},\quad \forall\,t\in(0,T),
$$
then from the estimates, we have the inequality:
$$
\Phi(t) \leq C_1 \int_0^t \Phi(s) \, ds + C_T\varepsilon^3.
$$
So, applying Gronwall's inequality gives 
$$
\Phi(t) \leq C_T \varepsilon^3 e^{C_1 t} \leq C_T \varepsilon^3.$$
All the above constant(s) are independent of $\e$ (but may depend on $T$). So, we obtain \eqref{Est01}.\\
This completes the proof.
  \end{proof}
		Now, we prove that a unique solution exists $\{U_\e,p_\e\}\in\GU_\e$ of the problem \eqref{MainPro}. For that, we construct a Galerkin approximation $U_n$ of $U_\e$ and $p_n$ of $p_\e$ and show its convergence to $\{U_\e,p_\e\}$ when $n\to+\infty$. This will justify the existence of unique solution $\{U_\e,p_\e\}\in\GU_\e$ of the problem \eqref{MainPro}.
		\begin{proof}[Proof of Theorem \ref{Th01}]
		{\bf Step 1.}	Observe that from \cite{evans}, we have that there exist an orthogonal basis $\{\xi_i\}_{i\in\N}$ of $\GV_\e$ with respect to the inner product
			$$(w,v)_{\GV_\e}=\int_{\O_\e}\GA_\e e(w)\,:\,e(u)\,dx,$$
			where 
			$\GV_\e=\{u\in H^1(\O_\e)^3\,|\, u=0,\quad\text{on $\Lambda_{0,\e}$}\}$.\\
        Similarly, we have an orthogonal basis $\{\eta_i\}_{i\in\N}$ of $H^1(\O^g_\e)$ which is an orthonormal  in $L^2(\O^g_\e)$.\\
			Then  the Galerkin approximation is
			\begin{equation}\label{52}
				\begin{aligned}
					\int_{\O_\e}\GA_\e e(U_n):e(v)\,dx-\int_{\O_\e^g}\alpha p_n\nabla\cdot v\,dx=\int_{\O_\e}f_\e\cdot v\,dx,\\
					\int_{\O_\e^g}(c\partial_tp_n+\alpha\nabla\cdot \partial_tU_n)\phi\,dx+\int_{\O_\e^g}\e^2K\nabla p_n\cdot \nabla\phi\,dx=\int_{\O_\e^g}h_\e\phi\,dx,\\
					U_n(0)=0,\quad \text{in $\O_\e$},\quad p_n(0)=0,\quad \text{in $\O^g_\e$},
				\end{aligned}
			\end{equation}
			for every $v\in \GV_n$, $\phi\in \GH_n$ and $t\in (0,T)$ where
			$$\begin{aligned}
				\GV_n=\text{span}\left\{\xi_1,\ldots,\xi_n\right\},\quad \GH_n=\text{span}\left\{\eta_1,\ldots,\eta_n\right\}.
			\end{aligned}$$
			Therefore,
			$$U_n=\sum_{i=1}^n a_i^n(t)\xi_i(x),\quad p_n=\sum_{i=1}^n b_i^n(t)\eta_i(x),$$
			with $a_{i}^n(0)=0$ and $b_i^n(0)=0$ for $i=1,\ldots,n$ due to the initial conditions on $U_n$ and $P_n$.\\
			Using the assumptions on $\GA_\e$ and Korn's inequality, we claim that the induced linear mapping $\BE_\e\,:\, \GV_\e\to H^{-1}(\O_\e)$ given by
			$$(\BE_\e(w),v)_{\GV_\e}=\int_{\O_\e}\GA_\e e(w)\,:\,e(v)\,dx,$$
			is a homeomorphism.     The linearity and continuity of $\BE_\e(w)$ are clear from the definition. We get the injectivity using the positive definiteness of $\GA_\e$ and Poincar\'e inequality. Using the Lax-Milgram theorem, we have the surjectivity and the continuity of the inverse.\\[1mm]
		{\bf Step 2.}	Let us define the following $n\X n$ matrices by 
			$$\GB_{ij}=\int_{\O_\e}\GA_\e e(\xi_i)\,.\,e(\xi_j)\,dx,\quad \GC_{ij}=\int_{\O^g_\e}\eta_i\nabla\cdot \xi_j\,dx,\quad \GD_{ij}=\e^2\int_{\O^g_\e}\nabla \eta_i\cdot\nabla\eta_j\,dx$$
			for $i,j\in \{1,\ldots,n\}$. Then, using the definition of $U_n$ and $p_n$, we obtain
			\begin{equation}\label{Ex02}
				\begin{aligned}
					\GB\Ga(t)-\alpha\GC\Gb(t)&=\GF_n(t),\\
					c{d\over dt}\Gb(t)+\alpha\GC^T{d\over \partial t}\Ga(t)+\GD\Gb(t)&=\GG_n(t),\quad \Gb(0)=0,
				\end{aligned}
			\end{equation}
			where $\Ga(t)=\left[a_1^n,\ldots,a_n^n(t)\right]^T$ and $\Gb(t)=\left[b_1^n,\ldots,b_n^n\right]^T$, with
			$$\GF_n(t)=\int_{\O_\e}f_\e\cdot U_n\,dx,\quad \GG_n(t)=\int_{\O^g_\e}h_\e p_n\,dx.$$
			Since, $\BE_\e$ is a homeomorphism, we have the invertibility of $\GB$, so after inserting \eqref{Ex02}$_1$ in \eqref{Ex02}$_2$, we obtain a linear system of ordinary differential equations for $\Gb$,
		\begin{equation}\label{ODE1}
			(c\GI_3+\alpha^2\GC^T\GB^{-1}\GC){d\over dt}\Gb(t)+\GD\Gb(t)=\GG_n-\alpha\GC^T\GB^{-1}{d\over dt}\GF_n(t),\quad \Gb(0)=0.
			\end{equation}
			Since $\BE_\e$ is a homeomorphism, we have coercivity for $\GB^{-1}$, so we get
            $$\Big((\GC^T\GB^{-1}\GC)\zeta,\zeta\Big)\geq k|\GC\zeta|^2_2,\quad\text{for all $\zeta\in\R^n$},$$
			which implies
            $$\Big(\big(c\GI_3+\alpha^2\GC^T\GB^{-1}\GC\big)\zeta,\zeta\Big)\geq k_0|\GC\zeta|^2_2,\quad\text{for all $\zeta\in\R^n$}.$$
           By applying the Lax–Milgram lemma in the Hilbert space setting (See Theorem 5 of Section 7.1 in \cite{evans}), it follows that \eqref{ODE1} there exists a unique weak solution $\Gb \in H^2(0, T)^n$. So, we have $H^2$-regularity in time for $U_n$ and $p_n$.\\[1mm]
			{\bf Step 3.} We derive a priori estimate uniform in $n$. Testing \eqref{52} by $(\partial_tU_n,p_n)$, we have the following using Step 1 of Lemma \ref{Lem06} , we have
            \begin{multline*}
                c\|p_n\|^2_{L^2(\O_\e^g)}+c_0\|e(U_n)\|^2_{L^2(\O_\e)}+c(\e)\int_0^t \|\nabla p_n\|^2_{L^2(\O^g_\e)}\,ds\\
                \leq 2\int_0^t \int_{\O^g_\e} h_\e p_n\,dx ds+\int_0^t \int_{\O_\e} f_\e\cdot \partial_t U_n\,dx ds.
            \end{multline*}
            Proceeding as in Step 2 of Lemma \ref{Lem06} along with Korn's inequality, we have
            \begin{multline*}
                \int_0^t\int_{\O_\e}f_\e\cdot\partial_tU_n\,dxds=-\int_0^t\int_{\O_\e}\partial_t f_\e\cdot U_n\,dxds+\int_{\O_\e}f_\e\cdot U_n\,dx\\
                \leq C(\e)\int_0^t \|\partial_tf_\e\|_{L^2(\O_\e)}\|e(U_n)\|_{L^2(\O_\e)}\,ds+C(\e)\|f_\e\|_{L^2(\O_\e)}\|e(U_n)\|_{L^2(\O_\e)}.
            \end{multline*}
            Now, using Young's inequality and proceeding as in Step 3 of Lemma \ref{Lem06}, we get
			$$\|U_n\|_{L^\infty(S;H^1(\O_\e))}+\|\nabla p_n\|_{L^2(S\X \O^g_\e)}+\|p_n\|_{L^\infty(S;L^2(\O^g_\e))}\leq C(T,\e).$$
   The constant $C(T,\e)$ depends on $T$ and $\e$ but is independent of $n$.\\
	So, passing to the limit $n\to +\infty$ in \eqref{52} gives the uniqueness and we have the $H^1$-regularity in time due to $H^2$-regularity in time for $U_n$ and $p_n$.
		\end{proof}
\section{Asymptotic analysis for all $t\in (0,T)$}\label{Sec07}
\subsection{The unfolding operators}
We achieve the convergences using the re-scaling unfolding operator $\Pi_\e$ for simultaneous homogenization dimension reduction in $\O_\e$ along with the unfolding operator $\Te$ for homogenization in $\o$. For detailed properties of the operators see \cite{CDG}.
\begin{definition}
Let $\psi$ be a measurable function on $\O_\e$, then the unfolding re-scaling operator $\Pi_\e$ is defined as 
	\begin{equation}\label{RUO1}
		\Pi_\e(\psi) (x',y) \doteq \begin{aligned}
			&\psi \left(\e\left[{x'\over \e}\right] + \e y\right) \qquad &&\hbox{for a.e. }\; (x',y)\in  \o\times \Yc.
		\end{aligned}
	\end{equation}
Similarly, let $\phi$ be a measurable function on $\o$, the periodic unfolding operator $\Te$ is defined as 
	$$ \Te(\phi) (x',y') \doteq \begin{aligned}
		&\psi \left(\e\left[{x'\over \e}\right] + \e y'\right) \qquad &&\hbox{for a.e. }\; (x',y')\in  \o\times Y.
	\end{aligned}$$
\end{definition} 
The operators $\Pi_\e$ and $\Te$ are  continuous linear operators  from $L^2(\O_\e)$ into $L^2(\o\X \Yc)$ and $L^2(\o)$ into $L^2(\o\X Y)$ respectively satisfying 
\begin{equation}\label{EQ722}
\begin{aligned}
    \|\Pi_\e(\psi)\|_{L^2(\o\X\Yc)}&\leq{C\over \sqrt{\e}}\|\psi\|_{L^2(\O_\e)}\quad &&\text{for every}\quad \psi\in L^2(\O_\e),\\
    \|\Te(\phi)\|_{L^2(\o\X Y)}&\leq C\|\phi\|_{L^2(\o)}\quad&&\text{for every}\quad \phi\in L^2(\o).
\end{aligned}
\end{equation}
The constant(s) $C$ do not depend on $\e$.\\
Note that for functions defined in $\o$, one has
	$$\Pi_\e(\phi)(x',y',0)=\Te(\phi)(x',y')=\phi \left(\e\left[{x'\over \e}\right] + \e y'\right)\quad \text{for every}\quad \phi\in L^2(\o).$$
Moreover, for every $\psi\in H^1(\O_\e)$ one has (see \cite[Proposition 1.35]{CDG})
\begin{equation}\label{UF01}
    \nabla_y\Pi_\e(\psi)(x',y)=\e\Pi_\e(\nabla \psi)(x',y) \quad \hbox{a.e. in}\quad \o\X\Yc.
\end{equation} 
    \subsection{Asymptotic limit of the strain tensor and pressure}
Let $(U_\e,p_\e)\in \GU_\e$ be the solution of the microscopic problem \eqref{MainPro} then, we have from \eqref{Est01} for all $t\in (0,T)$
$$ 	\|p_\e\|_{L^2(\O_\e^g))}+\|e(U_\e)\|_{L^2(\O^f_\e)^3)}+\|e(U_\e)\|_{L^2(\O_\e^g))}+\e\|\nabla p_\e\|_{L^2(\O_\e^g)}\leq C\e^{{3\over 2}}.$$
Using Lemma \ref{lem5}, we obtain $\Gw_\e$ and $\Gu_\e$, then using Lemma \ref{DeComKL1} and \ref{DeCom2}, we decompose $\Gw_\e$ and $\Gu_\e$ respectively.
As a consequence of the above estimates with Lemmas \ref{lem5}, \ref{DeComKL1}, \ref{DeCom2} and \ref{Lem06}, we have
\begin{equation}\label{71}
	\begin{aligned}
			\|e(\Gu_\e)\|_{L^2(\O_\e)}&\leq C\e^{{3/ 2}},\\
			\|\Uc_\e\|_{L^2(\o)}+\e\|\partial_\alpha \Uc_\e\|_{L^2(\o)}&\leq C\e^2,\\
			\|\wo{u}_\e\|_{L^2(\O_\e)}+\e\|\nabla\wo{u}_\e\|_{L^2(\O_\e)}&\leq C\e^{5/2},\\
			\|\Wc_{\alpha,\e}\|_{H^1(\o)}+\e(\|\Wc_{3,\e}\|_{H^1(\o)}+\|\Rc_{\alpha,\e}\|_{H^1(\o)})&\leq C\e,\\
			\|\partial_\alpha\Wc_{3,\e}+\Rc_{\alpha,\e}\|_{L^2(\o)}&\leq C\e,\\
			\|\wo{w}_\e\|_{L^2(\O_\e)} +\e\|\nabla\wo{w}_\e\|_{L^2(\O_\e)} &\leq C\e^{5/2},\\
			\|{p}_\e\|_{L^2(\O_\e^g)}+\e\|\nabla {p}_\e\|_{L^2(\O_\e^g)}&\leq C\e^{3/2}.
	\end{aligned}
\end{equation}
Then, from the above inequalities along \eqref{EQ722}--\eqref{UF01}, we get
\begin{equation}\label{72}
	\begin{aligned}
				\|\Pi_\e(\wo{u}_\e)\|_{L^2(\o\X\Yc)}+ \|\nabla_y\Pi_\e( \wo{u}_\e)\|_{L^2(\o\X\Yc)} &\leq C\e^2,\\
		\|\Pi_\e(\wo{w}_\e)\|_{L^2(\o\X\Yc)} +\|\nabla_y\Pi_\e(\wo{w}_\e)\|_{L^2(\o\X\Yc)}&\leq C\e^{2},\\
		\|\Pi_\e({p}_\e)\|_{L^2(\o\X\Yc^g)}+ \|\nabla_y\Pi_\e( {p}_\e)\|_{L^2(\o\X\Yc^g)} &\leq C\e.
	\end{aligned}
\end{equation}
We also have the following initial conditions
\begin{equation}\label{ICL01}
    \Wc_\e(0)=\Rc_\e(0)=\Uc_\e(0)=0,\quad\text{in $\o$}\quad\text{and}\quad \wo{w}_\e(0)=\wo{u}_\e(0)=0,\quad\text{in $\O_\e$}.
\end{equation}
Let us set
$$\begin{aligned}
	H^1_{per}(\Yc)&\doteq \left\{\phi\in H^1(\Yc)\,|\,\phi(y+n\Ge_\alpha)=\phi(y),\quad\text{a.e. in $\Yc$, $\forall\,n\in\Z$}\right\},\\
	\GH_{per}^1(\Yc)&\doteq\left\{\phi\in H^1_{per}(\Yc)|\phi=0\quad\text{a.e. in $\Gamma\cup {\Yc^f}$}\right\},\\
	H^1_{per,0}(\Yc)&\doteq \left\{\phi\in H_{per}^1(\Yc)\,|\,{1\over |\Yc|}\int_{\Yc}\phi\,dy=0\right\}.
\end{aligned}$$
\begin{lemma}\label{Lem08}
	There exist a subsequence of $\{\e\}$ (still denoted by $\{\e\}$) and $\Wc_\alpha\in H^1(\o)$, $\Wc_3\in H^2(\o)$ such that
	\begin{equation}\label{CFU1}
		\begin{aligned}
			{1\over \e}\Wc_{\alpha,\e}&\rightharpoonup \Wc_\alpha\quad &&\text{weakly in}\quad H^1(\o),\\
			\Wc_{3,\e}&\to \Wc_3\quad &&\text{weakly in}\quad H^1(\o),\\
		\Rc_{\alpha,\e} & \rightharpoonup -\partial_\alpha\Wc_3\quad&&\text{weakly in $H^1(\o)$},\\
			{1\over \e}\Uc_\e& \rightharpoonup 0\quad&&\text{weakly in $H^1(\o)^3$}. 
		\end{aligned}
	\end{equation}
				The fields	$\Wc_\alpha$, $\Wc_3$ satisfy the following boundary and initial conditions:
	$$ \Wc_\alpha=0\quad \text{a.e. on}\quad \partial\o,\quad \Wc_3=\partial_\alpha \Wc_3=0\quad\text{a.e. on}\;\; \partial\o,\quad\Wc(0)=0,\quad\text{in $\o$}.$$
	Moreover, there exist $\wh{W}_\alpha\in L^2(\o; H^1_{per,0}(Y))$, $\wh{R}_\alpha\in L^2(\o; H^1_{per,0}(Y))$ and $\wh U\in L^2(\o; H^1_{per,0}(Y))^3$ such that
	\begin{equation}\label{84}
		\begin{aligned}
			{1\over \e}\Te(\Wc_{\alpha,\e}) &\rightharpoonup \Wc_\alpha,\quad &&\text{weakly in $L^2(\o;H^1(Y))$},\\
			{1\over \e}\Te({\partial\Wc_{\beta,\e}\over \partial x_\alpha}) &\rightharpoonup {\partial \Wc_\beta\over \partial x_\alpha}+{\partial \wh{W}_\beta\over \partial y_\alpha},\quad &&\text{weakly in $L^2(\o\X Y)$},\\
			\Te(\Wc_{3,\e}) &\to \Wc_3,\quad &&\text{strongly in $L^2(\o;H^1(Y))$},\\
			\Te({\partial\Wc_{3,\e}\over \partial x_\alpha}) &\to {\partial \Wc_3\over \partial x_\alpha},\quad &&\text{strongly in $L^2(\o\X Y)$},\\
			\Te({\partial\Rc_{\beta,\e}\over \partial x_\alpha}) &\rightharpoonup -{\partial \Wc_3\over \partial x_\alpha\partial x_\beta}+{\partial \wh{R}_\beta\over \partial y_\alpha},\quad &&\text{weakly in $L^2(\o\X Y)$},\\
			{1\over \e}\Te({\partial\Uc_{\e}\over \partial x_\alpha}) &\rightharpoonup {\partial \wh{U}\over \partial y_\alpha},\quad &&\text{weakly in $L^2(\o\X Y)^3$}.
		\end{aligned}
	\end{equation}
	Furthermore, there exist $\wo{w} \in L^2(\o;H^1_{per}(\Yc))^3$, $p_0\in L^2(\o;H^1_{per}(\Yc^g))$ with ${p}_0(0)=0$ on $\o\X\Yc^g$ and ${u} \in L^2(\o;\GH_{per}^1(\Yc))^3$ such that
	\begin{equation}\label{CFfu1}
		\begin{aligned}
			{1\over \e^2}\Pi_\e(\wo{w}_\e) &\rightharpoonup \wo{w} \quad \text{weakly in}\quad L^2(\o;H^1(\Yc))^3,\\
			{1\over \e}\Pi_\e({p}_\e) &\rightharpoonup {p}_0 \quad \text{weakly in}\quad L^2(\o;H^1(\Yc^g)),\\
			{1\over \e^2}\Pi_\e\left(\wo{u}_\e\right) &\rightharpoonup u\quad\text{weakly in}\quad L^2(\o;H^1(\Yc))^3,
		\end{aligned}
	\end{equation}
	and
	\begin{equation}\label{CFgu1}
		\begin{aligned}
			{1\over \e}\Pi_\e\left(U_{\alpha,\e}\right)&\rightharpoonup \Wc_\alpha-y_3{\partial \Wc_3\over \partial x_\alpha}\quad &&\text{weakly in}\quad L^2(\o;H^1(\Yc)),\\
			\Pi_\e(U_{3,\e})&\to \Wc_3\quad &&\text{strongly in}\quad L^2(\o;H^1(\Yc)).
		\end{aligned}
	\end{equation}
	Finally, there exist  $\wo{\Gu}\in L^2(\o; H^1_{per,0}(\Yc))^3$ with $\wo{\Gu}(0)=0$ in $\o\X \Yc$ such that
		\begin{equation}\label{CFeu1}
		\begin{aligned}
			{1\over \e}\Pi_\e\left(e(U_\e)\right)&\rightharpoonup E(\Wc)+e_y(\wo{\Gu})\quad &&\text{weakly in}\quad L^2(\o\X \Yc)^{3\X3},\\
		  {1\over \e}\Pi_\e\left(\nabla\cdot U_\e\right)&\rightharpoonup \nabla\cdot \Wc_L+\nabla_y\cdot\wo{\Gu} \quad &&\text{weakly in}\quad L^2(\o\X \Yc).
		\end{aligned}
	\end{equation}
	where $E(\Wc)$ denotes the symmetric matrix
	$$
	E(\Wc)=
	\begin{pmatrix}
		e_{11}(\Wc)-y_3D_{11}(\Wc_3)&*&*\\
		e_{12}(\Wc)-y_3D_{12}(\Wc_3) &e_{22}(\Wc)-y_3D_{22}(\Wc_3) &*\\
		0&0&0
	\end{pmatrix},$$
    and
$$\Wc_L=(\Wc_1-y_3\partial_1\Wc_3)\Ge_1+(\Wc_2-y_3\partial_2\Wc_3)\Ge_2+\Wc_3\Ge_3.
$$
\end{lemma} 
\begin{proof}
{\bf Step 1.} The convergences \eqref{CFU1}, are the immediate consequences of the estimates \eqref{71}$_{2,3,4,5}$. The boundary and initial conditions satisfied by $\Wc$ come from the boundary and initial conditions satisfied by $\Wc_\e$ and $U_\e$.\\  The limit for the unfolded fields, i.e the convergences \eqref{84} is due to \eqref{CFU1} and using the properties of the unfolding operator given in Proposition 1.12, Theorem 1.36, Corollary 1.37 and Proposition 1.39 of \cite{CDG}.\\
The convergences \eqref{CFfu1} are the consequence of the estimates \eqref{72} and the properties of the re-scaling unfolding operator. Since we have
$\int_{-\e}^\e\wo{w}_\e\,dx_3=0$, which imply $\int_{-1}^1\wo{w}\,dy_3=0$ a.e. in $Y$. Similarly, since $\wo{u}_\e=0$ on $\Gamma_\e\cup{\O^f_\e}$ and $\int_{-\e}^\e \wo{u}_\e\,dx_3=0$, we get $u=0$ a.e. in $\Gamma\cup {\Yc^f}$ and $\int_{-1}^1\wo{u}\,dy_3=0$ a.e. in $Y$. Observe that $\wo{w}$ satisfy the following initial condition 
\begin{equation}\label{ICL02}
    \wo{w}(0)=u(0)=0,\quad\text{in $\o\X \Yc$}.
\end{equation}
{\bf Step 2.}
From Lemma \ref{DeComKL1} and \ref{DeCom2}, we have
\begin{equation}\label{EE01}
	\begin{aligned}
		U_{\alpha,\e}(\cdot,x_3)&=\Wc_{\alpha,\e}+x_3\Rc_{\alpha,\e}+\wo{w}_{\alpha,\e}(\cdot,x_3)+\Uc_{\alpha,\e}+\wo{u}_{\alpha,\e}(\cdot,x_3),\\
		U_{3,\e}(\cdot,x_3)&=\Wc_{3,\e}+\wo{w}_{3,\e}(\cdot,x_3)+\Uc_{3,\e}+\wo{u}_{3,\e}(\cdot,x_3),
	\end{aligned}\qquad \text{a.e. in}\quad (x',x_3)\in \O_\e.
\end{equation}
From the convergences \eqref{CFU1}, \eqref{CFfu1} with the estimates \eqref{71}, and using the re-scaling unfolding operator, we get using the fact that the terms with $\fu_\e$ vanish
$$\begin{aligned}
	{1\over \e}\Pi_\e(U_{\alpha,\e}) & \rightharpoonup \Wc_{\alpha}-y_3{\partial\Wc_{3}\over\partial x_\alpha}\quad \text{weakly in}\quad L^2(\o;H^1(\Yc)),\\
	\Pi_\e(U_{3,\e})&\to \Wc_3\quad\text{strongly in}\quad L^2(\o;H^1(\Yc)).
\end{aligned}$$
These give the convergences \eqref{CFgu1}.\\[1mm]
{\bf Step 3.} 
We have that
$$e(U_\e)=
\begin{aligned}
	&e(\Gw_\e)+e(\Gu_\e),\quad &&\text{a.e. in}\quad L^2(\O_\e)^{3\X3}.
\end{aligned}
$$
The strain tensor of $\Gw_\e$ is given by the following $3\X3$ symmetric matrix defined a.e. in $\O_\e$ by 
$$
\begin{aligned}
	e(\Gw_\e)&=e(W_{E,\e})+e(\wo{w}_\e)\\
	&=\begin{pmatrix}
		\ds e_{11}(\Wc_{\e})+x_3{\partial\Rc_{1,\e}\over\partial x_1}&*&*\\
		\ds e_{12}(\Wc_{\e})+x_3\left({\partial\Rc_{1,\e}\over\partial x_2}+{\partial\Rc_{2,\e}\over\partial x_1}\right)& \ds e_{22}(\Wc_{\e})+x_3{\partial\Rc_{2,\e}\over\partial x_2}&*\\
		{1\over 2}\Rc_{1,\e}&{1\over 2}\Rc_{2,\e}&0
	\end{pmatrix}+e(\wo{w}_\e).
\end{aligned}
$$
From the convergences \eqref{CFU1}, \eqref{CFfu1}$_{1,2}$ with definition of $e(\Gw_\e)$ we get
$$\begin{aligned}
	{1\over\e}\Pi_\e(e(U_\e))&={1\over\e}\Pi_\e\left(e(W_{E,\e})\right)+{1\over\e^2}e_y\left(\Pi_\e(\wo{w}_{\e})\right)+{1\over \e}\Pi_\e\big(e(\Uc_\e)\big)+{1\over \e^2}e_y\left(\Pi_\e(\wo{u}_\e)\right)\\
	&\rightharpoonup E(\Wc)+e_{y}(\wh{W}+y_3\wh R)+e_y(\wo{w})+e_y(\wh U)+e_y(u)\quad \text{weakly in}\quad L^2(\o\X \Yc)^{3\X3},
\end{aligned}$$
with $\wh W=\wh W_\alpha\Ge_\alpha$ and $\wh R=\wh R_\alpha\Ge_\alpha$.\\
Similarly, we get using the decomposition \eqref{EE01} and the convergence \eqref{84}--\eqref{CFfu1}$_{1,3}$ that
$${1\over \e}\Pi_\e\left(\nabla \cdot U_\e\right) \rightharpoonup \nabla \cdot \Wc_L+\nabla_y\cdot(\wh W+y_3\wh R) +\nabla_y \cdot \wo{w}+\nabla_y\cdot \wh U+\nabla_y\cdot u,\quad\text{weakly in $L^2(\o\X \Yc)$},$$
where
$$\Wc_L=(\Wc_1-y_3\partial_1\Wc_3)\Ge_1+(\Wc_2-y_3\partial_2\Wc_3)\Ge_2+\Wc_3\Ge_3.
$$
{\bf Step 4.}
Finally, we set the limit warping as follows
$$\wo{\Gu}=\wh{W}+y_3\wh R+\wo{w}+\wh U+u.$$
Hence, we have the convergences \eqref{CFeu1}. Observe that $\wo{\Gu}\in L^2(\o; H^1_{per,0}(\Yc))^3$ and $\wo{\Gu}(0)=0$ in $\o\X \Yc$ due to the properties satisfied by $\wh W$, $\wh R$, $\wo{w}$, $\wh U$ and $u$. This completes the proof.
\end{proof}
\section{Unfolded limit and homogenized problems}\label{Sec08}
In this section, we describe the limit problem and prove that it has a unique solution. For that, we define the limit displacement and pressure space by
$$\GU=\D\X H^1((0,T);L^2(\o;H^1_{per}(\Yc^g)),$$
where 
$$\D=\D_0\X H^1\big((0,T);L^2(\o; H_{per,0}^1(\Yc))^3\big),$$
with 
$$\begin{aligned}
\D_0&=\{\Wc=(\Wc_m,\Wc_3)\in H^1\big((0,T);H^1(\o)^2\X H^2(\o)\big)\;|\;\Wc=\partial_\alpha\Wc_3=0,\;\text{a.e. on $\partial\o$}\}.
\end{aligned}$$
Also, the limit fields $(\Wc,\wo{\Gu},p_0)\in \GU$ satisfy the following initial conditions
\begin{equation}\label{LIC01}
    \Wc(0)=\wo{\Gu}(0)=0,\quad\text{in $\o\X \Yc$}\quad\text{and}\quad p_0(0)=0,\quad\text{in $\o\X \Yc^g$}.
\end{equation}
The following is the main result of this paper:
\begin{theorem}\label{Th03}
    Let $\{(U_\e,p_\e)\}_\e$ be the sequence of unique solutions of the coupled system \eqref{MainPro}. Then, for the whole sequence $\{\e\}_\e$, we have for $t\in(0,T)$
    \begin{itemize}
        \item Convergence of displacement and pressure:
        \begin{equation}\label{IL02}
            \begin{aligned}
                	{1\over \e}\Pi_\e\left(U_{\alpha,\e}\right)&\rightharpoonup \Wc_\alpha-y_3{\partial \Wc_3\over \partial x_\alpha}\quad &&\text{weakly in}\quad L^2(\o;H^1(\Yc)),\\
			\Pi_\e(U_{3,\e})&\to \Wc_3\quad &&\text{strongly in}\quad L^2(\o;H^1(\Yc)),\\
                 {1\over \e}\Pi_\e(p_\e) &\rightharpoonup p_0,\quad&&\text{weakly in $L^2(\o; H^1(\Yc^g))$}.
            \end{aligned}
        \end{equation}
        \item Convergence of symmetric gradient and divergence of displacement: 
        \begin{equation}\label{IL01}
        \begin{aligned}
        {1\over \e}\Pi_\e\big(e(U_\e)\big)& \rightharpoonup E(\Wc)+e_y(\wo{\Gu}),\quad &&\text{weakly in $L^2(\o\X \Yc)^{3\X3}$},\\
        {1\over \e}\Pi_\e\left(\nabla\cdot U_\e\right)&\rightharpoonup \nabla\cdot \Wc_L+\nabla_y\cdot\wo{\Gu} \quad &&\text{weakly in}\quad L^2(\o\X \Yc).
    \end{aligned}
    \end{equation}
    \end{itemize}
     The limit field $(\Wc,\wo{\Gu},p_0)\in \D$ with  is the unique solution of the re-scaled unfolded problem:
    \begin{equation}\label{MUP}
        \begin{aligned}
         {1\over |\Yc|}\int_{\o\X \Yc}\GA\left(E(\Wc)+e_y(\wo{\Gu})\right):\left(E(\Vc)+e_y(\wo{\Gv})\right)\,dx'dy&-{1\over |\Yc|}\int_{\o\X \Yc^g}\alpha p_0\left(\nabla \cdot \Vc_L+\nabla_y\cdot\wo{\Gv}\right)\,dx'dy\\
         &\hskip 35mm =\int_{\o} f\cdot \Vc\,dx',\\
        \int_{\o\X\Yc^g}\partial_t\big(p_0+\alpha(\nabla\cdot\Wc_L+\nabla_y\cdot\wo{\Gu})\big)\phi\,dx'dy&+\int_{\o\X\Yc^g}K\nabla_yp_0\cdot\nabla_y\phi\,dx'dy\\
        &\hskip 35mm =\int_{\o\X\Yc^g}h\phi\,dx'dy	,
    \end{aligned}
    \end{equation}
    for $(\Vc,\wo{\Gv},\phi)\in \D$.\\[1mm]
    Moreover, the limit field $(\Wc,p_0)\in \D_0\X H^1((0,T); L^2(\o\; H^1_{per}(\Yc^g))$ is the unique solution of the homogenized problem: $\forall\,(\Vc,\phi)\in \D_0\X L^2(\o; H^1_{per}(\Yc^g))$
   	\begin{equation}\label{Hom01}
   		\left.\begin{aligned}
   			&\int_{\o} \GA^{hom}E(\Wc):E(\Vc)\,dx'
   			-\int_{\o}\alpha p_m\nabla\cdot \Vc_L\,dx'\\
   			&\hskip 12mm=\int_\o f\cdot \Vc\,dx'-\int_{\o\X \Yc}\GA e_y(\Gu_p):E(\Vc)\,dx'dy,\\
   			&\int_{\o\X\Yc^g}\partial_t\big(p_0+\alpha\nabla\cdot\Wc_L\big)\phi\,dx'dy
   			+\int_{\o\X\Yc^g}K\nabla_y p_0\cdot\nabla_y\phi\,dx'dy\\
   			&\hskip 12mm=\int_{\o\X \Yc^g}h\phi\,dx'dy-\alpha\int_{\o\X \Yc^g}\big(\nabla_y\cdot\partial_t\Gu_d-\nabla_y\cdot\partial_t\Gu_p)\,dx'dy
   		\end{aligned}\right\},
   	\end{equation}
   	where the homogenized tensor $\GA^{hom}$ is given by
       \begin{multline}\label{HomC02}    
    \GA^{hom}E(\Wc):E(\Vc)
    =\Big[a^{hom}_{\alpha\beta\alpha'\beta'}e_{\alpha\beta}(\Wc)e_{\alpha'\beta'}(\Vc)+b^{hom}_{\alpha\beta\alpha'\beta'}\left(e_{\alpha\beta}(\Wc)\partial_{\alpha'\beta'}\Vc_3+\partial_{\alpha\beta}\Wc_3e_{\alpha'\beta'}(\Vc)\right)\\
    +c^{hom}_{\alpha\beta\alpha'\beta'}\partial_{\alpha\beta}\Wc_3\partial_{\alpha'\beta'}\Vc_3\Big],
    \end{multline}
    with the homogenized coefficients given by    
    \eqref{HomC01} and $\Gu_d$ is given by \eqref{MD01} and $\Gu_p$ is the unique solution of the cell problem
   	$$ \int_{\Yc}\GA(y) e_y(\Gu_p) : e_y(\wo{\Gv})\,dy={1\over |\Yc|}\int_{\Yc^g}\alpha p_0\nabla_y\cdot \wo{\Gv}\,dy,\quad \text{for all $\wo{\Gv}\in H^1_{per}(\Yc)^3$},
   	$$
   	with
   	$$p_m(t,x')={1\over |\Yc|}\int_{\Yc^g}p_0(t,x',y)\,dy,\quad \Vc_L(t,x')=\sum_{\alpha=1}^2(\Vc_\alpha(t,x')-y_3\partial_\alpha\Vc_3(t,x'))\Ge_\alpha+\Vc_3(t,x')\Ge_3,$$
   	for $x'\in\o$ and $t\in (0,T)$.
\end{theorem}
\begin{proof}
The proof is presented in the subsequent steps.\\[1mm]
  {\bf Step 1.} We prove \eqref{IL02}--\eqref{IL01}.\\[1mm]
  Let $\{(U_\e,p_\e)\}_\e$ be the sequence of unique weak solutions of the problem \eqref{MainPro} in $\GU_\e$. Then, using Lemma \ref{Lem06} and the assumption on the force \eqref{AF01}, we have the estimates \eqref{Est01} for the fields $U_\e$ and $p_\e$. For a fix $\e>0$, we extend the field ${U_\e}_{|\O^f_\e}$ using Lemma \ref{lem5} to get $\Gw_\e$ and $\Gu_\e$ as in \eqref{ExDis}. Finally, we decompose $\Gw_\e$ and $\Gu_\e$ using Griso's decomposition technique in Lemma \ref{DeComKL1} and \ref{DeCom2}, respectively. So, as in Section \ref{Sec08}, we have estimates \eqref{71}--\eqref{72}. Finally, using the Lemma \ref{Lem08}, we have (at least for a subsequence) the convergences \eqref{IL02}--\eqref{IL01} with $(\Wc,\wo{\Gu},p_0)\in \GU$.\\[1mm] 
  {\bf Step 2.} We prove \eqref{MUP}$_1$.\\[1mm]
   Let us define the sequence of test functions.
    Let $v=(\Vc,\wo{\Gv})\in \D\cap \C^1((0,T);\C^1(\wo{\o})^2\X \C^2(\wo{\o})\X \C^1(\wo{\o}\X \wo{\Yc})^3)$. Then, we set
\begin{equation}\label{Test01}
\begin{aligned}
    	v_\e(t,x)&=\left[\e\Vc_\alpha(t,x')-x_3{\partial\Vc_3\over \partial x_\alpha}(t,x')\right]\Ge_\alpha+\Vc_3(t,x')\Ge_3+\e^2\wo{\Gv}\left(t,x',\left\{{x'\over \e}\right\},{x_3\over \e}\right),
\end{aligned}
\end{equation}
Then, similarly proceeding as in the proof of Lemma \ref{Lem08}, we obtain using the re-scaling unfolding operator, for all $t\in (0,T)$ 
\begin{equation}
\begin{aligned}
        {1\over \e}\Pi_\e\big(e(v_\e) &\to E(\Vc)+e_y(\wo{\Gv}),\quad &&\text{strongly in $L^2(\o\X \Yc)^{3\X3}$},\\
        {1\over \e}\Pi_\e\big(\nabla \cdot v_\e\big)& \to \nabla\cdot \Vc_L+\nabla_y\cdot\wo{\Gv},\quad&&\text{strongly in $L^2(\o\X \Yc)$}.
\end{aligned}
\end{equation}
with $\Vc_L=\left[\Vc_\alpha-y_3\partial_\alpha\Vc_3\right]\Ge_\alpha+\Vc_3\Ge_3$. We also have for the right-hand side
$$\lim_{\e\to0}{1\over \e^3}\int_{\O_\e}f_\e\cdot v_\e\,dx=\lim_{\e\to0}{1\over \e^2}\int_{\o\X \Yc}\Pi_\e(f_\e\cdot v_\e)\,dx'dy=|\Yc|\int_\o f\cdot \Vc\,dx'.$$
Taking $v_\e$ as a test function in \eqref{MainPro}$_1$, unfolding the equation with $\Pi_\e$ and dividing by $\e^3$ give
$$ \int_{\o\X \Yc}\GA{1\over \e}\Pi_\e\big(e(U_\e)\big):{1\over \e}\Pi_\e\big(e(v_\e)\big)\,dx'dy-\int_{\o\X \Yc^g}\alpha{1\over \e}\Pi_\e(p_\e){1\over \e}\Pi_\e\big(\nabla \cdot v\big)\,dx'dy={1\over \e^2}\int_{\o\X \Yc}\Pi_\e(f_\e\cdot v_\e)\,dx'dy,$$
then passing to the limit give \eqref{MUP}$_1$
for $(\Vc,\wo{\Gv})\in \D\cap \C^1((0,T);\C^1(\wo{\o})^2\X \C^2(\wo{\o})\X \C^1(\wo{\o}\X \wo{\Yc})^3)$.
Since $\D\cap \C^1((0,T);\C^1(\wo{\o})^2\X \C^2(\wo{\o})\X \C^1(\wo{\o}\X \wo{\Yc})^3)$ is dense in $\D$, we get \eqref{MUP}$_1$ for all $(\Vc,\wo{\Gv})\in\D$.\\[1mm]
{\bf Step 3.} We prove \eqref{MUP}$_2$.\\[1mm]
Let $\phi\in H^1((0,L); L^2(\o; H^1_{per}(\Yc^g))\cap \C^1((0,T);\C^1(\wo{\o}\X \wo{\Yc^g})$. Then, we set
$$\begin{aligned}
    \phi_{\e}(t,x)=\e\phi\left(t,x',\left\{{x'\over \e}\right\},{x_3\over \e}\right),
\end{aligned}\quad\text{a.e. in $\O^g_\e$}.$$
Observe that we have for all $t\in (0,T)$
{\begin{equation}\label{99}
    \begin{aligned}
    {1\over \e}\Pi_\e(\phi_\e) & \to \phi,\quad &&\text{strongly in $L^2(\o\X \Yc^g)$},\\
    \Pi_\e\big(\nabla \phi_\e\big)&\to \nabla_y\phi,\quad&&\text{strongly in $L^2(\o\X \Yc^g)$},
\end{aligned}
\end{equation}}
and for the right-hand side, we have
$$\lim_{\e \to0}{1\over \e^3}\int_{\O^g_\e}h_\e\phi_\e\,dx=\lim_{\e\to0}\int_{\o\X \Yc^g}\Pi_\e(h){1\over \e}\Pi_\e(\phi_\e)\,dx'dy=\int_{\o\X\Yc^g}h\phi\,dx'dy.$$
 Taking $\phi_\e$ as a sequence of test functions in \eqref{MainPro}$_2$, unfolding the equation with $\Pi_\e$ and dividing by $\e^3$ give
 \begin{multline*}
     \int_{\o\X \Yc^g}\partial_t\big(c{1\over \e}\Pi_\e(p_\e)+\alpha{1\over \e}\Pi_\e\big(\nabla\cdot U_\e\big)\big){1\over \e}\Pi_\e(\phi_\e)\,dx'dy+K\int_{\o\X \Yc^g}\Pi_\e (\nabla p_\e)\cdot \Pi_\e\big(\nabla\phi_\e\big)\,dx'dy\\
     =\int_{\o\X \Yc^g}\Pi_\e(h){1\over \e}\Pi_\e(\phi_\e)\,dx'dy.
 \end{multline*}
So,  by passing to the limit in the above expression along with the convergences \eqref{IL02}$_3$, \eqref{IL01}$_2$ give \eqref{MUP}$_2$ for $\phi\in H^1((0,T); L^2(\o; H^1_{per}(\Yc^g))\cap \C^1((0,T);\C^1(\wo{\o}\X \wo{\Yc^g})$. Using density argument, we have \eqref{MUP}$_2$ for $\phi\in H^1((0,T);L^2(\o; H^1_{per}(\Yc^g))$. \\[1mm]
{\bf Step 4.} In this step, we prove that there exist a unique solution to the unfolded limit problem \eqref{MUP}.\\
For that, it is enough to show that $f=0$ and $h=0$ give $(\Wc,\wo{\Gu},p_0)=0$ for all $(\Vc,\wo{\Gv},\phi)\in \D$.\\
We test the limit unfolded problem with $f=0$ and $h=0$ by $\Vc=\partial_t\Wc$, $\wo{\Gv}=\partial_t\wo{\Gu}$ and $\phi=p_0$. Then, integrating with respect to time and adding the two equations give
\begin{multline*}
        \int_0^t\int_{\o\X\Yc}\GA\left(E(\Wc)+e_y(\wo{\Gu})\right):\left(E(\partial_t\Wc)+e_y(\partial_t\wo{\Gu})\right)\,dx'dyds\\
        -\int_0^t\int_{\o\X \Yc^g}\alpha p_0\left(\nabla \cdot \partial_t\Wc_L+\nabla_y\cdot\partial_t\wo{\Gu}\right)\,dx'dyds
         +\int_0^t\int_{\o\X\Yc^g}\partial_t\big(p_0+\alpha(\nabla\cdot\Wc_L+\nabla_y\cdot\wo{\Gu})\big)p_0\,dx'dyds\\
         +\int_{0}^t\int_{\o\X\Yc^g}K\nabla_yp_0\cdot\nabla_y p_0\,dx'dyds=0,
\end{multline*}
which imply
\begin{multline*}
    \int_0^t\int_{\o\X \Yc^g}\partial_t p_0p_0\,dx'dyds+\int_0^t\int_{\o\X\Yc}\GA\left(E(\Wc)+e_y(\wo{\Gu})\right):\left(E(\partial_t\Wc)+e_y(\partial_t\wo{\Gu})\right)\,dx'dyds\\
    +\int_0^t K\|\nabla_yp_0\|^2_{L^2(\o\X \Yc^g)}\,ds=0.
\end{multline*}
Proceeding as in Step 1 of Lemma \ref{Lem06} along with the coercivity of $\GA$ and the fact that $K$ is a positive definite, we get
\begin{multline*}
c\|p_0(t)\|^2_{L^2(\o\X \Yc^g)}+c_0\left\|\big(E(\Wc)+e_y(\wo{\Gu})\big)(t)\right\|^2_{L^2(\o\X \Yc}+2c_K\int_0^t\|\nabla_y p_0(s)\|^2_{L^2(\o\X \Yc^g)}\,ds\\
\leq c\|p_0(0)\|^2_{L^2(\o\X \Yc^g)}+ c_0\left\|\big(E(\Wc)+e_y(\wo{\Gu})\big)(0)\right\|^2_{L^2(\o\X \Yc)}.
\end{multline*}
Using the initial conditions \eqref{LIC01} and the Korn-type inequality \eqref{UK01}, we have for $t\in (0,T)$
$$\|p_0(t)\|^2_{L^2(\o\X\Yc^g))}+\sum_{\alpha=1}^2\|\Wc_\alpha(t)\|^2_{H^1(\o)}+\|\Wc_3(t)\|^2_{H^2(\o)}+\|\wo{\Gu}(t)\|^2_{L^2(\o; H^1(\Yc))}+\|\nabla_y p_0\|^2_{L^2((0,T)\X \o\X \Yc^g)}=0,$$
which give $(\Wc,\wo{\Gu},p_0)=(0,0,0)$.\\[1mm]
{\bf Step 5.} Using the uniqueness of the solution $(\Wc,\wo{\Gu},p_0)\in\GU$, we have that all the above convergences \eqref{IL02}--\eqref{IL01} hold for the whole sequence.\\[1mm]
{\bf Step 6.}
Now, we derive the homogenized problem for that we express the microscopic displacement $\wo{\Gu}$ in terms of $\Wc$, $p_0$ and some correctors using cell problems.\\
Let us chose $\Vc=0$ and localizing the Problem \eqref{MUP}$_1$, we obtain
\begin{equation}\label{ML01}
    \int_{\Yc}\GA(y)\big(E(\Wc)+e_y(\wo{\Gu})\big):e_y(\wo{\Gu})\,dy-\int_{\Yc^g}\alpha p_0\nabla_y\cdot\wo{\Gv}\,dy=0,\quad\forall\; \wo{\Gv}\in H^1_{per,0}(\Yc)^3.
\end{equation}
So, we have a linear problem and using Korn's inequality we have a unique solution. We also have that the microscopic displacement $\wo{\Gu}$ can be written in terms of the tensor $E(\Wc)$, $p_0$ and some correctors.\\
We define the following $3\X 3$ symmetric matrices by
	$$\GM^{ij}_{kl}={1\over 2}\left(\d_{ki}\d_{lj}+\d_{kj}\d_{li}\right),\quad i,j,k,l\in\{1,2,3\},$$
	where $\d_{ij}$ is the Kronecker symbol.\\
    So, we have the following expression for $\wo{\Gu}$
    \begin{equation}\label{MD01}
    \begin{aligned}
                \wo{\Gu}(t,x',y)&=\Gu_d(t,x',y)-\Gu_p(t,x',y)\\
                &=\sum_{\alpha,\beta=1}^2\left(e_{\alpha\beta}(\Wc)(t,x')\chi^{\alpha\beta}_m(y)+\partial_{\alpha\beta}\Wc_3(t,x')\chi^{\alpha\beta}_b(y)\right)-\Gu_p(t,x',y)
    \end{aligned},\quad\text{for a.e. $y\in\Yc$},
    \end{equation}
    where correctors $\chi^{\alpha\beta}_m$ and $\chi^{\alpha\beta}_b$  are the unique solutions of the cell problems
    \begin{equation*}
    \begin{aligned}
         &\text{Find $(\chi^{\alpha\beta}_m,\chi^{\alpha\beta}_b)\in [H^1_{per,0}(\Yc)^3]^6$ and }\\
         &\hskip 5mm \left.\begin{aligned}
             \int_{\Yc}\GA_{ijkl}(y)\big(\GM^{\alpha\beta}_{ij}+e_{y,ij}(\chi^{\alpha\beta}_m)\big) e_{y,kl}(\wo{\Gv})\,dy=0,\\
             \int_{\Yc}\GA_{ijkl}(y)\big(y_3\GM^{\alpha\beta}_{ij}+e_{y,ij}(\chi^{\alpha\beta}_b)\big) e_{y,kl}(\wo{\Gv})\,dy=0
         \end{aligned}\right\},\quad \text{for all $\wo{\Gv}\in H^1_{per,0}(\Yc)^3$},\\
    \end{aligned}
    \end{equation*}
    The correctors $\Gu_p$ is the unique solution of the following 
    \begin{equation*}
        \int_{\Yc}\GA(y) e_y(\Gu_p) : e_y(\wo{\Gv})\,dy={1\over |\Yc|}\int_{\Yc^g}\alpha p_0\nabla_y\cdot \wo{\Gv}\,dy,\quad \text{for all $\wo{\Gv}\in H^1_{per,0}(\Yc)^3$}.
    \end{equation*}
    We set the homogenized coefficients in $\Yc$ as 
    \begin{equation}\label{HomC01}
    \begin{aligned}
        a^{hom}_{\alpha\beta\alpha'\beta'}&={1\over |\Yc|}\int_{\Yc}\GA_{ijkl}(y)\left[\GM^{\alpha\beta}_{ij}+e_{y,ij}(\chi^{\alpha\beta}_m)\right]\GM^{\alpha'\beta'}_{kl}\,dy,\\
        b^{hom}_{\alpha\beta\alpha'\beta'}&={1\over |\Yc|}\int_{\Yc}\GA_{ijkl}(y)\left[y_3\GM^{\alpha\beta}_{ij}+e_{y,ij}(\chi^{\alpha\beta}_m)\right]\GM^{\alpha'\beta'}_{kl}\,dy,\\
        c^{hom}_{\alpha\beta\alpha'\beta'}&={1\over |\Yc|}\int_{\Yc}\GA_{ijkl}(y)\left[y_3\GM^{\alpha\beta}_{ij}+e_{y,ij}(\chi^{\alpha\beta}_m)\right]y_3\GM^{\alpha'\beta'}_{kl}\,dy.
    \end{aligned}
    \end{equation}
    Then, using the above-homogenized coefficients and the expression of the microscopic displacement \eqref{MD01} along with equation \eqref{ML01}, we get new expression of \eqref{MUP}$_1$ as
    \begin{equation*}\label{HLP01}
        \int_{\o} \GA^{hom}E(\Wc):E(\Vc)\,dx'
        -\int_{\o\X \Yc}\GA e_y(\Gu_p):E(\Vc)\,dx'dy-{1\over |\Yc|}\int_{\o\X \Yc^g}\alpha p_0 \nabla \cdot \Vc_L\,dx'dy=\int_\o f\cdot \Vc,dx',
    \end{equation*}
    where $\GA^{hom}E(\Wc):E(\Vc)$ is given by \eqref{HomC02}.\\  
   Introducing the expression \eqref{MD01} of $\wo{\Gu}$ in \eqref{MUP}$_2$ give
   \begin{multline*}
       \int_{\o\X\Yc^g}\partial_t\big(p_0+\alpha\nabla\cdot\Wc_L\big)\phi\,dx'dy+\alpha\int_{\o\X \Yc^g}\big(\nabla_y\cdot\partial_t\Gu_d-\nabla_y\cdot\partial_t\Gu_p)\,dx'dy\\
       +\int_{\o\X\Yc^g}K\nabla_y p_0\cdot\nabla_y\phi\,dx'dy=\int_{\o\X \Yc^g}h\phi\,dx'dy.
   \end{multline*}
   As a consequence of the above expression, we get \eqref{Hom01} from \eqref{MUP} and the rest follows from the above steps.\\   
This completes the proof.
\end{proof}
\begin{remark}
    The convergences \eqref{IL02}$_{1,2}$ show that the limit displacement is of {\bf Kirchhoff-Love} type.
\end{remark}
\section*{Conclusion and future outlook}
In conclusion, this paper deals with simultaneous homogenization and dimension reduction of a microscopic model of a fiber-reinforced hydrogel. More precisely, the FIH consists of a periodic connected elastic fiber-scaffold and a disconnected hydrogel part model via Biot linear poro-elasticity. On the interface standard continuity of displacement, stress and flux is assumed. An asymptotically based homogenized limit model is derived when both the period and thickness are of the same order and tend to zero simultaneously. The limit model is the Biot-Kirchhoff-Love plate model, which satisfies the corresponding homogeneous boundary and initial conditions as the microscopic model.\\[1mm]
In reality, the fiber and hydrogel components are interconnected. However, in the current draft (where the hydrogel part is disconnected), we analyze the hydrogel part in each cell as a separate bounded domain with a Lipschitz boundary. When the hydrogel is connected, the boundary of the hydrogel part in each periodicity cell may not be Lipschitz. This complicates the process of deriving estimates and passing to the limit. Therefore, the case of both components being connected will be addressed in future work.

Additionally, in the current draft, we assume the periodicity and thickness are of the same order, but the thickness of the fiber-scaffold can be smaller than the size of an individual cell. For that, it would also be interesting to consider the thickness parameter as $0<\d$ with the assumption, $\ds\lim_{(\d,\e)\to(0,0)}\ds{\d\over \e}=0$, where $0<\e$ is the periodicity parameter.

\section*{Acknowledgment}
The first author would like to thank Prof.\ Dr.\ Georges Griso and Dr.\ Julia Orlik for their guidance and supervision. The second author gratefully acknowledges the PMRF (ID-2402788) scheme and IIT Kharagpur for funding his PhD position.
 
\bibliography{reference}
\appendix
\section{Korn-type inequality for unfolded fields}
\begin{lemma}
    For $(\Wc,\wo{\Gu})\in \D$, there exist $C>0$ such that
    \begin{equation}\label{UK01}
        \sum_{\alpha=1}^2\|\Uc_\alpha\|^2_{H^1(\o)}+\|\Uc_3\|^2_{H^2(\o)}+\|\wo{\Gu}\|^2_{L^2(\o; H^1(\Yc))}\leq C\|E(\Uc)+e_y(\wo{\Gu})\|^2_{L^2(\o\X \Yc)}.
    \end{equation}
\end{lemma}
The above inequality is a direct consequence of  $2$D-Korn's inequality and a norm equivalence from \cite{Amartya}, which states
\begin{lemma}[Lemma 16, \cite{Amartya}]\label{NEq1}
	Consider the space $\S=\R^3\X\R^3\X H^1_{per,0}(\Yc)^3$ with the semi-norm 
	$$	\|(\eta,\zeta,w)\|_\S = \sqrt{\sum_{i,j=1}^3  \|\wt{\cal E}_{ij}(\eta,\zeta,w)\|^2_{L^2( \Yc)}},
	$$ 
	where for every $(\eta,\zeta,w) \in \S$, we denote $\wt{\cal E}$ the symmetric matrix by
	$$
	\wt{\cal E}(\eta,\zeta,w)=
	\begin{pmatrix}
		\eta_1-y_3\zeta_1+ e_{11,y}( w)  & \eta_3 -y_3 \zeta_3 + e_{12,y}( w) 
		&  e_{13,y}( w)  \\
		* & \eta_2-y_3\zeta_2 + e_{22,y}( w)  &  e_{23,y}( w) \\
		* & *&    e_{33,y}( w) 
	\end{pmatrix}.
	$$ 
	Then, there exists constants $c,C>0$ such that for all $(\eta,\zeta,w) \in \S$ we have
	\begin{equation}\label{ENE1}
		c\left(\|\eta\|^2_2 + \|\zeta\|^2_2 + \|w\|^2_{H^1(\Yc)}\right) \leq  \|(\eta,\zeta,w)\|^2_\S
		\leq C \left(\|\eta\|^2_2 + \|\zeta\|^2_2 + \|w\|^2_{H^1(\Yc)}\right).
	\end{equation} 
\end{lemma}

\end{document}